\documentclass[a4paper]{scrartcl}
\usepackage[english]{babel}
\usepackage[utf8]{inputenc}					
\usepackage[T1]{fontenc}							
\usepackage{amsfonts}									
\usepackage{newcent}									
\usepackage{fouriernc}								
\usepackage{dsfont}										
\usepackage{helvet}										
\usepackage{inconsolata}
\usepackage{enumitem}
\usepackage{blkarray}
\usepackage{amsrefs}
\usepackage{amssymb}
\usepackage{amsthm}
\usepackage{latexsym}
\usepackage{stmaryrd}
\usepackage{mathtools}
\usepackage{hyperref}
\usepackage{todonotes}
\usepackage{verbatim}
\usepackage{thmtools}
\usepackage{booktabs}
\usepackage{float}
\usepackage{caption}
\usepackage{subcaption}

\theoremstyle{definition}
\newtheorem{definition}{Definition}
\newtheorem{theorem}[definition]{Theorem}
\newtheorem{lemma}[definition]{Lemma}
\newtheorem{cor}[definition]{Corollary}
\newtheorem{example}[definition]{Example}
\newtheorem{remark}[definition]{Remark}
\newtheorem{manualtheoreminner}{Hypothesis}
\newenvironment{manualtheorem}[1]{
  \renewcommand\themanualtheoreminner{#1}
  \manualtheoreminner
}{\endmanualtheoreminner}

\DeclareMathOperator{\terr}{\operatorname{terr}}
\DeclareMathOperator{\sym}{\operatorname{Sym}}
\DeclareMathOperator{\wrg}{\wr_{\Gamma}}
\DeclareMathOperator{\supp}{\operatorname{supp}}
\DeclareMathOperator{\fix}{\operatorname{fix}}
\DeclareMathOperator{\Yade}{\operatorname{Yade}}
\DeclareMathOperator{\symg}{\operatorname{Sym}(\Gamma)}
\DeclareMathOperator{\ld}{\operatorname{ld}}
\renewcommand{\P}{\mathcal{P}}
\renewcommand{\L}{\mathcal{L}}
\DeclareMathOperator{\C}{\mathcal{C}}
\renewcommand{\ker}{\operatorname{ker}}
\DeclareMathOperator{\im}{\operatorname{im}}
\newcommand\restr[2]{{            
\left.\kern-\nulldelimiterspace 
#1                              
\vphantom{\big|}                
\right|_{#2}                    
}}
\newcommand\myrestriction[2]{{    
\left.\kern-\nulldelimiterspace 
#1                              
\vphantom{\big|}                
\right|_{#2}^{\Gamma}           
}}
\makeatletter
\def\moverlay{\mathpalette\mov@rlay}
\def\mov@rlay#1#2{\leavevmode\vtop{
\baselineskip\z@skip \lineskiplimit-\maxdimen
\ialign{\hfil$\m@th#1##$\hfil\cr#2\crcr}}}
\newcommand{\charfusion}[3][\mathord]{
#1{\ifx#1\mathop\vphantom{#2}\fi
	\mathpalette\mov@rlay{#2\cr#3}
	}
\ifx#1\mathop\expandafter\displaylimits\fi}
\makeatother

\newcommand{\cupdot}{\charfusion[\mathbin]{\cup}{\cdot}}
\newcommand{\bigcupdot}{\charfusion[\mathop]{\bigcup}{\cdot}}

\newcommand{\zset}[1]{\{0,\dotsc, #1 \}}
\makeatletter
\newcommand*\bigcdot{\mathpalette\bigcdot@{0.6}}
\newcommand*\bigcdot@[2]{\mathbin{\vcenter{\hbox{\scalebox{#2}{$\m@th#1\bullet$}}}}}
\makeatother
\newcommand{\zero}{\bigcdot} 
\usepackage{abstract}
\newcommand{\highlight}[3]{\mathchoice%
  {\colorbox{#1}{\color{#2}$\displaystyle#3$}}%
  {\colorbox{#1}{\color{#2}$\textstyle#3$}}%
  {\colorbox{#1}{\color{#2}$\scriptstyle#3$}}%
  {\colorbox{#1}{\color{#2}$\scriptscriptstyle#3$}}}%
\definecolor{colorZ}	{RGB}{ 221, 221, 221 } 
\definecolor{colorA}    {RGB}{ 221, 204, 119 } 
\definecolor{colorB}    {RGB}{ 204, 102, 119 } 
\definecolor{colorC}    {RGB}{ 136, 204, 238 } 
\definecolor{colorD}    {RGB}{  68, 170, 153 } 
\definecolor{colorE}    {RGB}{ 153, 153,  51 } 
\definecolor{colorF}    {RGB}{ 136,  34,  85 } 
\definecolor{colorG}    {RGB}{  51,  34, 136 } 
\definecolor{colorH}    {RGB}{  17, 119,  51 } 

\newcommand{\highlightZ}[1]{\highlight{colorZ!30}{colorZ!50!black}{#1}}
\newcommand{\highlightA}[1]{\highlight{colorA!30}{colorA!50!black}{#1}}
\newcommand{\highlightB}[1]{\highlight{colorB!30}{colorB!50!black}{#1}}
\newcommand{\highlightC}[1]{\highlight{colorC!30}{colorC!50!black}{#1}}
\newcommand{\highlightD}[1]{\highlight{colorD!30}{colorD!50!black}{#1}}
\newcommand{\highlightE}[1]{\highlight{colorE!40}{colorE!10!white}{#1}}
\newcommand{\highlightF}[1]{\highlight{colorF!40}{colorF!10!white}{#1}}
\newcommand{\highlightG}[1]{\highlight{colorG!40}{colorG!10!white}{#1}}
\newcommand{\highlightH}[1]{\highlight{colorH!40}{colorH!10!white}{#1}}

\author{Dominik Bernhardt, Alice C. Niemeyer,\\ Friedrich Rober, Lucas Wollenhaupt}
\title{Conjugacy classes and centralisers in wreath products}
\begin{document}
\maketitle
\begin{abstract}
In analogy to the disjoint cycle decomposition in permutation groups, Ore and Specht
define a decomposition of elements of the full monomial group and exploit
this to describe conjugacy classes and centralisers of elements in the full monomial group.\\
We generalise their results to wreath products whose base group need not be
finite and whose top group acts faithfully on a finite set.
We parameterise conjugacy classes and
centralisers of elements in such wreath products explicitly.
For finite wreath products, our approach yields efficient algorithms
for finding conjugating elements, conjugacy classes, and centralisers.
\end{abstract}
\section{Introduction}
Wreath product constructions feature prominently in the
theory of permutation groups, see for example
\cites{Baddeley_Transitive_Simple, Kovacs} and references therein.
A general approach to working with
permutation groups
is first to apply a reduction to primitive groups and subsequently 
to study primitive groups. Many problems for primitive groups
are solved with the help of the O'Nan-Scott Theorem, see for example
\cites{LPS, DM}, in which most classes are defined via wreath products.

To our knowledge, Specht was the first to describe the conjugacy classes of
the full monomial group, namely $K\wrg \symg$ for a finite set
$\Gamma$, see \cite[S\"atze II, III, IV]{Specht}.
Specht and Ore define  a
wreath cycle decomposition for elements of full monomial groups in analogy to
the  disjoint  cycle  decomposition   of  elements  of  the  symmetric
group. Moreover, Ore gives criteria when two elements  are conjugate
in the group and determines the centraliser of an element, see \cite[Theorem 8]{Ore}.\\
In order to make Specht's and Ore's theory more widely known, we have
decided to restate their results in the modern language of wreath products.
Moreover, we have extended the results to the more general setting of
arbitrary wreath products $K\wr_\Gamma H$. We no longer require $H$ to be the
full symmetric group on $\Gamma$, however, we assume that $H$ acts
faithfully on $\Gamma$. For a detailed description of the groups we
investigate, see Hypothesis \ref{setting}.\\
We describe the wreath cycle decomposition and
define a \emph{territory decomposition} for an element of $K\wrg H$, see 
Definition \ref{def:territory_decomposition}, which generalises the
\emph{type} of a
wreath product element defined by Specht \cite[(8b)]{Specht}.
Viewing elements in wreath products in
a disjoint wreath cycle decomposition quickly becomes very  intuitive  and
highlights much of the underlying structure. For example, it is just
as  easy  to read  off  the  order of  an  element  in disjoint wreath cycle
decomposition in a wreath product as it is to read off the order of an
element from its disjoint cycle  decomposition in the symmetric group.
Additionally, it is a very efficient way of representing elements of wreath
products on a computer.

The main aim of this paper is to prove several results
that can be summarised as follows:
\begin{theorem}\label{main-thm-1}
Two elements $w=(f,h)$ and $v=(e,g)$ of $K\wrg \symg$
are conjugate in $K\wrg H$ if and only if
there exists an element $t\in H$ that conjugates $h$ to $g$ and maps the
territory decomposition of $w$ to that of $v$. Moreover, there exists an
explicit, computable bijection from an iterated cartesian product 
into the conjugacy classes of $K\wrg H$.
\end{theorem}
\begin{theorem}\label{main-thm-2}
There exists an explicit, computable bijection from an iterated cartesian product into
the centraliser of an element of $K\wrg \symg$ in $K\wrg H$. Moreover, this 
centraliser is an extension of two groups. 
\end{theorem}
The statements of our main theorems facilitate efficient
computation of centralisers, conjugacy classes of elements and conjugacy 
testing on a computer. This approach has been implemented in the \textsf{GAP}-package
\textit{WPE} \cite{Rober_Package} by the third author. For example,
one is now able to test elements for conjugacy and to
compute conjugating elements in groups as large as $S_{25}\wr S_{100}$ in a
few seconds. For further computational results, see Section \ref{section-6}.

Cannon and Holt describe algorithms in \cite{Holt_Cannon} to compute
centralisers and
conjugacy classes of elements and perform conjugacy testing in a finite
group with trivial soluble radical. They embed the
given group into a direct product of certain wreath products and
solve these tasks for each direct factor. 
Hulpke, see \cite{Hulpke_cc}, presents an algorithm to compute the conjugacy
classes in finite permutation groups in which he considers a more general
situation of subdirect products of the base group.
Compared to both \cite{Holt_Cannon} and \cite{Hulpke_cc}, our methods are
further reaching for wreath products as we
exploit the underlying \emph{wreath cycle decomposition}.
We translate the explicit descriptions in Theorems \ref{main-thm-1} and \ref{main-thm-2}
into very efficient algorithms in practice, see Section
\ref{section-6}. We hope that our methods could be used to improve the algorithms of
Cannon and Holt when treating the wreath products that occur as direct factors.
\subsection{The structure of this paper}
In Section 2 in Theorem
\ref{thm:wreathcycle-decomposition} we restate, in the modern language of
wreath products, Ore's decomposition of a wreath product element
generalising the decomposition of a permutation into disjoint cycles.
In Section 3 we first define the territory decomposition of a
wreath product element, see Definition \ref{def:territory_decomposition}. Next, we
give a criterion to decide whether two elements of $K\wrg \symg$ are conjugate in 
$K\wrg H$ and, if they are, construct a conjugating element,
see Theorem \ref{conjugationInWreathProd}. 
Theorem \ref{thm-cc} in Section 4 parameterises the conjugacy classes of
$K\wrg H$ explicitly. Theorem \ref{main-thm-1} immediately follows from 
Theorems \ref{conjugationInWreathProd} and \ref{thm-cc}. Theorem 
\ref{thm:structure_centraliser} parameterises the centraliser of a wreath
product element in $K\wrg H$ explicitly. Together with Theorem
\ref{extension_centraliser} this immediately implies Theorem \ref{main-thm-2}.
As some of  the results may seem quite technical  on first reading, we
have  illustrated  them with  several  examples. The final section
gives evidence of the computational power of our results.

\section{Wreath Cycle Decompositions in Wreath Products}
The aim of this section is to generalise the concept of a disjoint cycle
decomposition of permutations to arbitrary elements of wreath products. The
main statement of this section is Theorem
\ref{thm:wreathcycle-decomposition} in which
we give an explicit decomposition of wreath product elements as a product of
disjoint wreath cycles. This result is well known, see for example
Ore \cite{Ore} or Kerber et. al. \cite[Section 4.2]{Kerber}. We adapt this theorem
to our notation and give an explicit constructive formula for the
decomposition of a wreath product element.
For a permutation $h\in \symg$, we denote by $\supp_\Gamma(h)$ its support on $\Gamma$, i.e.
the set of points in $\Gamma$ moved by $h$. The set of 
fixed points of a permutation $h$ is denoted by $\fix_\Gamma(h)$.
If the choice of $\Gamma$ is clear from the context we omit it. For the entire paper, 
we fix the following setting:
\begin{manualtheorem}{A}\label{setting}
Let $K$ be a not necessarily finite group, $\Gamma$ a finite set and $H\leq \sym(\Gamma)$.
Further, set $W \coloneqq K\wr_\Gamma H \coloneqq K^\Gamma \rtimes H =
\{(f,h) \,:\, f \in K^\Gamma,\, h \in H\}$ and denote by $S \coloneqq K\wr_\Gamma
\sym(\Gamma)$ the full monomial group (on $K$ with respect to $\Gamma$).
We denote the set of functions from $\Gamma$ to $K$ by
$K^\Gamma$ and apply functions from the right, i.e. we write $[\gamma]f$
for the image of $\gamma\in \Gamma$ under $f\in K^\Gamma$. 
Accordingly, all groups act from the right. 
\end{manualtheorem}

We start by extending the concepts of support and cycles in permutation groups
to arbitrary wreath products. \emph{Wreath cycles} were already introduced
by Ore \cite{Ore} who called them \emph{monomial cycles}.
\begin{definition}\label{Definition-Wreath-Cycle}
Let \(w \coloneqq (f,h),\, v \coloneqq (e, g) 
\in K \wrg H \).
\begin{enumerate}[label=\arabic*.]
\item The element $h\in H$ is called the
\textbf{top component} of $w=(f,h)$ and the element $f\in K^\Gamma$ is called the
\textbf{base component} of $w$.
\item We define \(\terr_{\Gamma}(w) \coloneqq \supp_{\Gamma}(h) \cup
\{\gamma \in \Gamma \,:\, [\gamma]f \neq 1_K\}\). We call \(\terr_{\Gamma}(w)\) the 
\emph{territory} of \(w\). If the choice of \(\Gamma\) is clear from
context or of no importance, we write \(\terr(w)\) instead.
\item \label{def:terr} We call \(w=(f,h)\) a \emph{wreath cycle} if either \(h\) induces the
identity  on \(\Gamma\) and \(\vert \terr(w) \vert = 1\) or \(h\)
induces a single non-trivial cycle in its action on \(\Gamma\) and \(\terr(w) = \supp(h)\).
\item We say \(w\) and \(v\) are \emph{disjoint} if
\(\terr(w) \cap \terr(v) = \emptyset\).
\end{enumerate}
\end{definition}
If $\Gamma= \{1,\dots,n\}\subseteq \mathbb{Z}_{>0}$,
we denote an element $w\coloneqq (f,h)\in W$ as $(f,h)=(f_1,\dots,f_n;h)$,
where for $i\in \Gamma$ the element $f_i\coloneqq [i]f$ is called the $i$-th
\emph{base component} of $w$. For $v\coloneqq (e,g) = (e_1,\dots, e_n;g)\in
W$, multiplication is given by \[w\cdot v = \left(f\cdot e^{h^{-1}}\!,\, h\cdot g\right) =
\left(f_1\cdot e_{1^h},\dots, f_n\cdot e_{n^h}; h\cdot g\right)\] and hence inverses
are given by \[w^{-1} = \left(\left(f^{-1}\right)^h,h^{-1}\right) =
\left(\left(f_{1^{h^{-1}}}\right)^{-1}, \dots, \left(f_{n^{h^{-1}}}\right)^{-1};
h^{-1}\right).\]
\begin{example}
\label{Example-Wreath-Product-Territory}
Let $K \coloneqq \sym(\{1, \dots, 4\})$,
$\Gamma \coloneqq \{1,\dots, 8\}$ and
$S \coloneqq K \wr \sym(\Gamma)$.
Throughout this paper the element 
\[
w \coloneqq \big(
\overset{\highlightC{1}}{(1,2)(3,4)}, \hspace*{0.5em}
\overset{\highlightC{2}}{(3,4)},\hspace*{0.5em}
\overset{\highlightZ{3}}{()},\hspace*{0.5em}
\overset{\highlightC{4}}{(1,2),\hspace*{0.5em}}
\overset{\highlightC{5}}{(1,2,3)},\hspace*{0.5em}
\overset{\highlightZ{6}}{()},\hspace*{0.5em}
\overset{\highlightC{7}}{(1,2)},\hspace*{0.5em}
\overset{\highlightZ{8}}{()};\hspace*{0.5em}
\overset{\text{top}}{(\highlightB{1,2})(\highlightB{3,4})(\highlightB{5,6})}
\big)
\]
with $w = (f,h)\in S$ is used in all examples. The territory of $w$ is
\[\terr(w)
= \{\highlightB{1,2,3,4,5,6}\} \cup
\{\highlightC{1,2,4,5,7}\}
= \{1,2,3,4,5,6,7\}.\]
In Lemma \ref{disjointElemCommute} we prove that elements with
disjoint territories commute,
which can be viewed of as an extension of the fact that permutations
with disjoint support commute.\\
Now consider the following two elements of $S$:
\[
\begin{array}{l*{10}{c}}
v \coloneqq \bigl(
& \overset{\highlightZ{1}}{()},
& \overset{\highlightZ{2}}{()},
& \overset{\highlightZ{3}}{()},
& \overset{\highlightZ{4}}{()},
& \overset{\highlightZ{5}}{()},
& \overset{\highlightZ{6}}{()},
& \overset{\highlightC{7}}{(1,2)},
& \overset{\highlightZ{8}}{()};
& \overset{\text{top}}{()}
&\bigr), \\
u \coloneqq \bigl(
& \overset{\highlightC{1}}{(1,2)(3,4)},
& \overset{\highlightC{2}}{(3,4)},
& \overset{\highlightZ{3}}{()},
& \overset{\highlightZ{4}}{()},
& \overset{\highlightZ{5}}{()},
& \overset{\highlightZ{6}}{()},
& \overset{\highlightZ{7}}{()},
& \overset{\highlightZ{8}}{()};
& \overset{\text{top}}{(\highlightB{1,2})}
&\bigr).
\end{array}
\]
These elements are examples of the two different types of wreath
cycles in part 3 of Definition \ref{Definition-Wreath-Cycle}.
For $v$, the top component acts trivially on $\Gamma$
and $\vert \terr(v) \vert = \vert \{7\} \vert = 1$.
For $u$, we observe that the top component of $u$
induces the single cycle $(1,2)$ on $\Gamma$
and $\terr(u) = \{1,2\} = \supp((1,2))$.
\end{example}
It is sometimes useful to write the territory of a wreath product element
as a disjoint union of the support of the top component and the fixed points of
the top component that are contained in the territory of the wreath product
element. This will be used throughout this paper in several proofs.
\begin{remark}\label{decomposition-territory}
Let $w \coloneqq (f,h)\in K\wr_\Gamma H$. Then
$\terr(w) = \supp(h) \cupdot \left(\fix(h)\cap \terr(w)\right)$.
\end{remark}
The next lemma shows that disjoint wreath product elements commute.

\begin{lemma}\label{disjointElemCommute} 
Let \(w = (f,h)\) and \(v = (e,g)\) be two disjoint elements of \(K
\wr_\Gamma H\). Then \(w\) and \(v\) commute.
\end{lemma}

\begin{proof}
Let \(w = (f,h)\) and \(v = (e,g)\) be disjoint. We first show 
\[fe=ef,\quad e^h=e \text{ and } f^g=f.\]
As $g$ and $h$ are disjoint, 
it follows that $[\gamma]f = 1_K$ or $[\gamma]e = 1_K$ and
thus we have $[\gamma](f\cdot e) = [\gamma](e\cdot f)$ for any $\gamma \in \Gamma$.
Further, for all $\gamma \in \supp(h)$ we have $\gamma,\, \gamma^h \not \in \terr(v)$ and
therefore $[\gamma]e^{h^{-1}} = [\gamma^h]e = 1_K = [\gamma]e$.
Additionally, if $\gamma \not \in \supp(h)$ we have
$[\gamma]e^{h^{-1}} = [\gamma^h]e = [\gamma]e$.\\
This, together with a similar argument for $f^{g^{-1}}$,
shows that \[fe^{h^{-1}} = fe = ef = ef^{g^{-1}}\]
Furthermore, \(h\) and \(g\) are 
disjoint permutations of $\Gamma$ and therefore commute. Thus
\[(f,h)(e,g) = (fe^{h^{-1}},hg) = (ef^{g^{-1}},gh) = (e,g)(f,h).\qedhere\]
\end{proof}
Analogously to a disjoint cycle decomposition for permutations we define
a wreath cycle decomposition of a wreath product element into disjoint wreath cycles.
Just as the individual cycles in a disjoint
cycle decomposition of a permutation need not be elements
of the group, the wreath cycles in the disjoint
wreath cycle decomposition need not be elements of the wreath product.
\begin{definition}
A \emph{wreath cycle decomposition} for a wreath product element $w\in W =
K\wrg H$
is a decomposition of $w$ as $w = \prod_{i=1}^\ell w_i$ where the $w_i\in
S = K\wrg \symg$ are pairwise disjoint wreath cycles for all $i=1,\dots, \ell$.
\end{definition}
Our next aim is to give a formula to compute a wreath cycle decomposition and
to show that it is unique up to ordering of the factors.
Throughout this paper, the following function simplifies the definition 
of certain elements. In particular, it is used in the construction of a
disjoint wreath cycle decomposition.
\begin{definition}
Let $\Omega \subseteq \Gamma$. For a map $f:\Gamma\to K$ we
define 
\[
\myrestriction{f}{\Omega}: \Gamma \to K, \gamma \mapsto
\begin{cases}
[\gamma]f,& \text{if } \gamma \in \Omega\\
1_K,& \text{else}.
\end{cases}
\]
For simplicity, we set $\myrestriction{f}{\gamma}
\coloneqq \myrestriction{f}{\{\gamma\}}$ for $\gamma \in \Gamma$.
\end{definition}
Note that the function $\myrestriction{f}{\Omega}$ agrees with $f$ on all
of $\Omega$ and maps the elements of $\Gamma\setminus \Omega$ to $1_K$.

\begin{theorem}\label{thm:wreathcycle-decomposition} 
Every element of \(K\wrg H\) can be written as a finite product of disjoint 
wreath cycles in $S$. This decomposition is unique up to ordering of the
factors.
\end{theorem}

\begin{proof}
Let $w \coloneqq (f,h) \in K\wrg H$. We prove the first statement of this theorem by giving an
explicit wreath cycle decomposition of $w$.\\
Let $h = h_1 \cdots h_\ell$ be the disjoint cycle decomposition of $h$ in $\sym(\Gamma)$.
We claim that the following is the desired decomposition of $w$:
\[
w = \prod_{i=1}^\ell(\myrestriction{f}{\supp(h_i)},h_i) \quad \cdot \!
\prod_{\gamma \in \fix(h) \cap \terr(w)} (\myrestriction{f}{\gamma}, 1_H).
\]
Note that all factors are disjoint wreath cycles because their territories
are pairwise disjoint. We now show equality by using Lemma
\ref{disjointElemCommute}, i.e. the fact that disjoint wreath
product elements commute:
\begin{equation}\label{eq:existence}
\begin{aligned}
&\prod_{i=1}^\ell(\myrestriction{f}{\supp(h_i)},h_i) \quad \cdot \!
	\prod_{\gamma \in \fix(h) \cap \terr(w)} (\myrestriction{f}{\gamma}, 1_H)\\
	=\quad &\left(\prod_{i=1}^\ell(\myrestriction{f}{\supp(h_i)},1_H) 
	\cdot (1_{K^\Gamma}, h_i)\right)
	\quad \cdot \! \prod_{\gamma \in \fix(h) \cap \terr(w)} (\myrestriction{f}{\gamma}, 1_H)\\
	=\quad& \left(\prod_{i=1}^\ell ( \myrestriction{f}{\supp(h_i)},1_H) \quad \cdot \!
	\prod_{\gamma \in \fix(h) \cap \terr(w)} (\myrestriction{f}{\gamma},
	1_H) \right) \cdot	\prod_{i=1}^\ell (1_{K^\Gamma}, h_i)\\
	=\quad & (f,1_H)\cdot (1_{K^\Gamma}, h) = (f,h) = w.
\end{aligned}
\end{equation}
To prove uniqueness, suppose $w = \prod_{i=1}^m (e_i,g_i)$ is 
another disjoint wreath cycle
decomposition of $w$. Thus $h=g_1\cdots g_m$ is a disjoint cycle
decomposition of $h$ and we may assume $g_1,\dots, g_k \neq 1_H$ and
$g_{k+1},\dots ,g_m = 1_H$ for some $1\leq k\leq m$.
As the $g_i$ are pairwise disjoint and due to the uniqueness of the
cycle decomposition in $\symg$ we have $k=\ell$ and
there exists a permutation $\sigma \in \sym(\ell)$ such that 
$h_i = g_{i^\sigma}$ for all $i\in \{1, \dots, \ell\}$.
Hence, without loss of generality we may assume
$h_i=g_i$ for all $i\in \{1, \dots, \ell\}$. By a calculation along the
lines of computation (\ref{eq:existence}) above, we obtain $\prod_{i=1}^m e_i = f$. For a
given $\gamma \in \terr(w)$ with $[\gamma]f \neq 1_K$, there exists a
unique $j\in \{1,\dots,m\}$ such that $[\gamma]e_j = [\gamma]f$.\\
First assume $\gamma \in \supp(h_i)$ for some 
$i\in \{1,\dots, \ell\}$.
Note 
\[ \gamma \in \terr((e_j,g_j)) = \supp(g_j) \text{ and } \gamma \in \supp(h_i) =
\terr\left((\myrestriction{f}{\supp(h_i)},h_i)\right)\]
and since $h=g_1\cdots g_\ell = h_1 \cdots h_\ell$ are two cycle decompositions of the same
permutation $h$ we must have $i=j$. 
Since the $(e_j,g_j)$ are pairwise disjoint wreath cycles for all $j=1,\dots, m$ and $f=e_1\cdots
e_m$ we obtain $e_i
= \myrestriction{f}{\supp(h_i)}$ for all $i\in \{1,\dots, \ell\}$.\\
Now we assume
$\gamma \in \fix(h)\cap \terr(w)$. As $\gamma \in
\fix(h)$, we have $j>\ell$. As $g_j = 1_H$ and $ \vert
\terr((e_j,g_j)) \vert =1$ we have $e_j = \myrestriction{f}{\gamma}$.
\end{proof}

\begin{example}
\label{Example-Wreath-Product-Wreath-Cycle-Decomp}
Let $K \coloneqq \sym(\{1, \dots, 4\})$,
$\Gamma \coloneqq \{1,\dots, 8\}$ and
$S \coloneqq K \wr \sym(\Gamma)$.
We give an example of a wreath cycle
decomposition by considering the element $w = (f,h) \in S$
from Example \ref{Example-Wreath-Product-Territory}:
\[
w \coloneqq \big(
 \overset{\highlightB{1}}{(1,2)(3,4)},\hspace*{0.5em}
\overset{\highlightB{2}}{(3,4)},\hspace*{0.5em}
 \overset{\highlightC{3}}{()},\hspace*{0.5em}
 \overset{\highlightC{4}}{(1,2)},\hspace*{0.5em}
 \overset{\highlightD{5}}{(1,2,3)},\hspace*{0.5em}
 \overset{\highlightD{6}}{()},\hspace*{0.5em}
 \overset{\highlightA{7}}{(1,2)},\hspace*{0.5em}
 \overset{\highlightZ{8}}{()};\hspace*{0.5em}
 \overset{\text{top}}{(\highlightB{1,2})(\highlightC{3,4})(\highlightD{5,6})}
\big).
\]
We obtain the decomposition by applying the steps of the proof of Theorem
\ref{thm:wreathcycle-decomposition} to $w$.
We know $\terr(w) = \{1,2,3,4,5,6,7\}$ and $\supp(h) = \{1,2,3,4,5,6\}$.
A disjoint cycle decomposition of $h$ is
$h = \highlightB{h_1}
\cdot \highlightC{h_2}
\cdot \highlightD{h_3}$
with $\highlightB{h_1} = (\highlightB{1,2})$,
$\highlightC{h_2} = (\highlightC{3,4})$ and
$\highlightD{h_3} = (\highlightD{5,6})$.
Next we compute base components in $K^\Gamma$
for each cycle in the disjoint cycle decomposition of $h$, which are given by
$\highlightB{f_1} \coloneqq \myrestriction{f}{\supp(\highlightB{h_1})}$,
$\highlightC{f_2} \coloneqq \myrestriction{f}{\supp(\highlightC{h_2})}$
and $\highlightD{f_3} \coloneqq \myrestriction{f}{\supp(\highlightD{h_3})}$
As an example, we consider
\[
	\highlightB{f_1} = \myrestriction{f}{\supp(\highlightB{h_1})}: \Gamma \to K, \gamma \mapsto
	\begin{cases}
		[\gamma]f, & \text{if } \gamma \in \highlightB{\{1,2\}}\\
		(),& \text{else}
	\end{cases}
\]
and compute some images under this map explicitly.
The image of $1$ under $f_1$ is $[1]\highlightB{f_1} = [1]f =
(1,2)(3,4)$ and $[3]\highlightB{f_1} = ()$.
Note $\fix(h)\cap \terr(w) = \{\highlightA{7}\}$, so the only
remaining point for which we need to compute a base component in $K^\Gamma$
is the point $7$ which is given by $\highlightA{f_4} \coloneqq
\myrestriction{f}{\highlightA{\{7\}}}$.
	We set $\highlightA{h_4} = ()$ and $w_i = (f_i, h_i)$ for $1 \leq i \leq 4$.
This yields a wreath
cycle decomposition of $w$ into the product of the following four
wreath cycles in $S$:
\[
\arraycolsep=0.45em\def\arraystretch{1.8}
\begin{array}{l*{10}{c}}
\highlightB{w_1} = \bigl(
& \overset{\highlightB{1}}{(1,2)(3,4)},
& \overset{\highlightB{2}}{(3,4)},
& \overset{\highlightZ{3}}{()},
& \overset{\highlightZ{4}}{()},
& \overset{\highlightZ{5}}{()},
& \overset{\highlightZ{6}}{()},
& \overset{\highlightZ{7}}{()},
& \overset{\highlightZ{8}}{()};
& \overset{\text{top}}{(\highlightB{1,2})}
&\bigr), \\
\highlightC{w_2} = \bigl(
& \overset{\highlightZ{1}}{()},
& \overset{\highlightZ{2}}{()},
& \overset{\highlightC{3}}{()},
& \overset{\highlightC{4}}{(1,2)},
& \overset{\highlightZ{5}}{()},
& \overset{\highlightZ{6}}{()},
& \overset{\highlightZ{7}}{()},
& \overset{\highlightZ{8}}{()};
& \overset{\text{top}}{(\highlightC{3,4})}
&\bigr), \\
\highlightD{w_3} = \bigl(
& \overset{\highlightZ{1}}{()},
& \overset{\highlightZ{2}}{()},
& \overset{\highlightZ{3}}{()},
& \overset{\highlightZ{4}}{()},
& \overset{\highlightD{5}}{(1,2,3)},
& \overset{\highlightD{6}}{()},
& \overset{\highlightZ{7}}{()},
& \overset{\highlightZ{8}}{()};
& \overset{\text{top}}{(\highlightD{5,6})}
&\bigr), \\
\highlightA{w_4} = \bigl(
& \overset{\highlightZ{1}}{()},
& \overset{\highlightZ{2}}{()},
& \overset{\highlightZ{3}}{()},
& \overset{\highlightZ{4}}{()},
& \overset{\highlightZ{5}}{()},
& \overset{\highlightZ{6}}{()},
& \overset{\highlightA{7}}{(1,2)},
& \overset{\highlightZ{8}}{()};
& \overset{\text{top}}{()}
&\bigr).
\end{array}
\]
\end{example}

\section{Solving the conjugacy problem in wreath products}
Recall the setting from Hypothesis \ref{setting}.
The main result of this section, Theorem \ref{conjugationInWreathProd},
gives a solution to the \emph{conjugacy problem} for wreath product
elements: Given two wreath product elements $w$ and $v$ of $S$, decide
whether they are conjugate in $W$ and give a conjugating element if it
exists. Theorem \ref{main-thm-1} is an immediate consequence.\\
The wreath cycle decomposition plays a crucial role in the solution of the
conjugacy problem. It turns out that it suffices to define many concepts
only for wreath cycles and then apply them to every cycle in the wreath cycle decomposition.
We first define a property of a wreath cycle which is invariant under
conjugation by elements of $S$. For this we need the following function.
\begin{definition}\label{def:yade}
Define $\C(W) \coloneqq \{w \in W \, :\, w \text{ is a wreath cycle}\}$ as
the set of all wreath cycles in $W$ and the \emph{Yade-map} by
\[ [-,-]\Yade \colon \Gamma \times \C(W) \to K,~
(\gamma,(f,h)) \mapsto \prod \limits_{i=0}^{\vert h
\vert -1} [\gamma]\left(f^{h^{-i}}\right) = \prod_{i=0}^{\vert h
\vert -1} \left[\gamma^{h^i}\right]f\,.\]
Given a wreath cycle \(w = (f,h)\) and \(\gamma \in \Gamma\) we call $[\gamma,w]\Yade$
the \emph{Yade} of \(w\) in $\gamma$.
\end{definition}
Yade stands for \textit{Yet another determinant} since, after choosing a suitable
matrix representation of $K\wr_\Gamma \sym(\Gamma)$, the Yade-map 
can be interpreted as a matrix determinant whence Ore called it a determinant in
\cite[p.\ 19]{Ore}. James and
Kerber also introduce this map in \cite[Section 4.3]{Kerber} and call it a
\emph{cycle product}.\\
Observe that \([\gamma,w]\Yade = 1_K\) for all \(\gamma \not \in \terr(w)\)
for a wreath cycle $w\in S$. In computations involving the Yade-map,
it is therefore enough to consider the restriction $\restr{[-,w]\Yade}{\terr(w)}$.

Suppose that $w$ is a wreath cycle. The next lemma shows that
$[\alpha,w]\Yade$ and $[\beta,w]\Yade$ are $K$-conjugate whenever $\alpha$
and $\beta$ are contained in the territory of $w$. We also give a
conjugating element explicitly.
\begin{lemma}\label{conjugateDet} 
Let \(w=(f,h) \in K \wrg H\) be a wreath cycle. Then \([\alpha,w]\Yade\) and
\([\beta,w]\Yade\) are conjugate in \(K\) for every \(\alpha, \beta \in \terr(w)
\). For some \(j \in \zset{\vert h \vert - 1}\) with
$\alpha^{h^j}=\beta$, we have
\[\left([\alpha,w]\Yade\right)^{y} = [\beta,w]\Yade,\quad
\text{where } y = \prod_{i=0}^{j -1} \left[ \alpha^{h^i}\right]f.\]
\end{lemma}

\begin{proof}
As $w$ is a wreath cycle we prove this statement by distinguishing two
cases according to part $3$ of Definition \ref{Definition-Wreath-Cycle}:
Either $h=1_H$ and $\terr(w)$
is a singleton or $h$ induces a single
non-trivial cycle on $\Gamma$ and $\terr(w)=\supp(h)$.\\
First, consider the case of $h=1_H$. Then $\alpha=\beta$ since $\terr(w)$
is a singleton and the claim becomes trivial. Now consider the case of $h$
being a non-trivial cycle.
Let \(\alpha, \beta \in \supp(h)\). Then we have 
\[a \coloneqq [\alpha,w]\Yade = \prod \limits_{i=0}^{\vert h \vert -1}
[\alpha]f^{h^{-i}} = \prod \limits_{i=0}^{\vert h \vert -1} [\alpha^{h^{i}}]f\]
and
\[ b \coloneqq [\beta,w]\Yade = \prod \limits_{i=0}^{\vert h \vert -1} 
[\beta]f^{h^{-i}} = \prod \limits_{i=0}^{\vert h \vert -1}
[\beta^{h^i}]f.\]
Since \(h\) is a cycle there is a \(j \in \zset{\vert h \vert -1}\) such 
that \(\alpha^{h^j} = \beta\). A straight forward computation shows that for 
$y \coloneqq \prod_{i=0}^{j -1} \left[ \alpha^{h^i}\right]f$ we
have $b = a^y$ thus \(a\) and \(b\) are conjugate in \(K\) and the claim follows.
\end{proof}
The previous lemma justifies the following definition.
\begin{definition}
Let $w=(f,h)\in S$ be a wreath cycle and let $\alpha \in
\terr(w)$ be fixed. We call the conjugacy class
$\left([\alpha,w]\Yade\right)^K$ of $[\alpha,w]\Yade$ in
$K$ the \emph{Yade-class of $w$}.
\end{definition}
The following example shows how to compute the Yade-map.
\begin{example}
\label{Example-Yade}
Let $K \coloneqq \sym(\{1, \dots, 4\})$,
$\Gamma \coloneqq \{1,\dots, 8\}$ and
$S \coloneqq K \wr \sym(\Gamma)$.
Consider the wreath cycles from Example \ref{Example-Wreath-Product-Territory}:
\[
\begin{array}{r*{10}{c}}
v \coloneqq \bigl(
& \overset{\highlightZ{1}}{()},
& \overset{\highlightZ{2}}{()},
& \overset{\highlightZ{3}}{()},
& \overset{\highlightZ{4}}{()},
& \overset{\highlightZ{5}}{()},
& \overset{\highlightZ{6}}{()},
& \overset{\highlightC{7}}{(1,2)},
& \overset{\highlightZ{8}}{()};
& \overset{\text{top}}{()}
&\bigr), \\
u \coloneqq (e, g) = \bigl(
& \overset{\highlightB{1}}{(1,2)(3,4)},
& \overset{\highlightB{2}}{(3,4)},
& \overset{\highlightZ{3}}{()},
& \overset{\highlightZ{4}}{()},
& \overset{\highlightZ{5}}{()},
& \overset{\highlightZ{6}}{()},
& \overset{\highlightZ{7}}{()},
& \overset{\highlightZ{8}}{()};
& \overset{\text{top}}{(\highlightB{1,2})}
&\bigr).
\end{array}
\]
We compute the Yade of $u$ in $1$:
\[
	[1, u]\Yade = \prod_{i=0}^{\vert g
	\vert -1} \left[1^{g^i}\right]e = [1]e \cdot [2]e = (1,2)(3,4) \cdot (3,4) = (1,2).
\]
Then
\[
\begin{array}{l*{9}{c}}
[-,v]\Yade \coloneqq \bigl(
& \overset{\highlightZ{1}}{()},
& \overset{\highlightZ{2}}{()},
& \overset{\highlightZ{3}}{()},
& \overset{\highlightZ{4}}{()},
& \overset{\highlightZ{5}}{()},
& \overset{\highlightZ{6}}{()},
& \overset{\highlightC{7}}{(1,2)},
& \overset{\highlightZ{8}}{()}
&\bigr), \\[0pt]
[-,u]\Yade \coloneqq \bigl(
& \overset{\highlightB{1}}{(1,2)},
& \overset{\highlightB{2}}{(1,2)},
& \overset{\highlightZ{3}}{()},
& \overset{\highlightZ{4}}{()},
& \overset{\highlightZ{5}}{()},
& \overset{\highlightZ{6}}{()},
& \overset{\highlightZ{7}}{()},
& \overset{\highlightZ{8}}{()}
&\bigr).
\end{array}
\]
\end{example}
The following corollary shows how to compute the order of a wreath cycle as a
product of the order of its top component and of an element in $K$. This
and the following lemma are stated in Ore, see \cite[Theorem 4]{Ore}. Note
that by $\vert w\vert$ we denote the order of $w\in W$.
\begin{cor}\label{detordercor}
Let $w=(f,h)\in S$ be a wreath cycle.
Then for any $\gamma \in \terr(w)$ we have 
\[\lvert w \rvert = \bigl\vert[-,w]\Yade\bigr\vert \cdot \vert h\vert =\bigl\vert
[\gamma,w]\Yade \bigr\vert \cdot \vert h\vert,\]
where $\bigl\vert[-,w]\Yade\bigr\vert$ is the order of $[-,w]\Yade$ as an
element of $K^\Gamma$.
\end{cor}

\begin{proof}
By the definition of the Yade-map and the multiplication rule in wreath
products we have $w^{\vert h\vert}=([-,w]\Yade, 1_H)$,
thus $\vert w \vert = \bigl\vert[-,w]\Yade\bigr\vert \cdot \vert h\vert	$.
It remains to show that for any $\gamma \in \terr(w)$ we have
$\bigl\vert[-,w]\Yade\bigr\vert =\bigl\vert[\gamma,w]\Yade\bigr\vert$.
This follows immediately
by the pointwise multiplication in $K^\Gamma$ and the fact that
the non-trivial images of $[-,w]\Yade$ are conjugate in $K$ by
Lemma~\ref{conjugateDet} and thus all have the same order in $K$.
\end{proof}

We can generalize this observation to arbitrary wreath product elements.
This can be seen as an analogue to computing orders of permutations given
in disjoint cycle decomposition. Recall that $\operatorname{LCM}$ denotes
the least common multiple.
\begin{lemma}\label{orderlemma}
Let $w=w_1\dotsm w_\ell \in W$ be a wreath cycle decomposition of $w$ 
with $w_i=(f_i,h_i)\in S$.
Then \[\vert w\vert =\operatorname{LCM}\left(\vert w_1\vert , \dots,\, \vert
w_\ell\vert \right)= \operatorname{LCM}\left(\bigl\vert [-,w_1]\Yade \bigr\vert \cdot \vert
h_1\vert,\dots,\, \bigl\vert[-,w_\ell]\Yade\bigr\vert \cdot \vert h_\ell\vert\right).\]
\end{lemma}

\begin{proof}
The elements $w_i$ and $w_j$ commute pairwise for
all $1\leq i,\, j \leq \ell$ and therefore $|w|=\operatorname{LCM}(\vert w_1\vert,\dots, \vert
w_\ell\vert)$. The result follows by Corollary~\ref{detordercor}.
\end{proof}
We now turn our attention to the conjugacy problem in wreath products and first
consider the case of wreath cycles. It turns out that a generalization of
the length of a cycle, which we call the load, will be very useful. Recall
that by Lemma \ref{conjugateDet} we know that $[\alpha,w]\Yade$ and
$[\beta,w]\Yade$ are conjugate in $K$ for all $\alpha,\, \beta \in \terr(w)$.
\begin{definition}
Let $w = (f,h) \in W$ be a wreath cycle and let $\alpha \in \terr(w)$. 
We define the  \emph{load} of $w$ as
the tuple \[\operatorname{ld}(w)\coloneqq 
\left(\, \left([\alpha,w]\Yade\right)^K,\vert h\vert \right).\]
\end{definition}
We now prove that the load of a wreath cycle is invariant under
conjugation in $S$.
\begin{lemma}\label{conjugateInWreath}
Let \(w = (f,h) \in S\) be a wreath cycle.
For every \(a = (s,t)\in S\) the conjugate \(a^{-1}wa = w^a\) is a 
wreath cycle such that for each $\gamma \in \Gamma$ we have
\[
\left[\gamma\!,\,w^a\right]\Yade = \left[\gamma^{t^{-1}}\right]
s^{-1}\cdot \left[\gamma^{t^{-1}}\!,\,w\right]\Yade \cdot \left[\gamma^{t^{-1}}\right]s.\]
In particular, $\ld(w)=\ld(w^a)$.
\end{lemma}

\begin{proof}
Observe that
\[
w^a = (f,h)^{(s,t)} = \left(\left(s^{-1}\right)^{t} f^t s^{h^{-1} t}\!,\, t^{-1} h t \right).
\]
We first show that $w^a$ is a wreath cycle. For this, first consider
the case of $h = 1_H$. Then
\[
w^a = \left((f^s)^t, 1_H \right).
\]
Thus the top component of $w^a$ is trivial. Since
for every $\gamma \in \Gamma$ we have
\[
[\gamma](f^s)^t =
\left[\gamma^{t^{-1}}\right]s^{-1}\cdot 
\left[\gamma^{t^{-1}}\right]f\cdot \left[\gamma^{t^{-1}}\right]s
\]
and as $w$ is a wreath cycle with trivial top component, 
$w$ has exactly one non-trivial base component and we obtain $|\terr(w^a)| = 1$.\\
Now let us consider $h \neq 1_H$.
The top component $h$ is a non-trivial cycle and hence $h^t$ is a non-trivial cycle.
Further for any $\gamma \notin \supp(h^t)$ we know $\gamma^{t^{-1}} \notin \supp(h)$ and $\gamma^{t^{-1} h} = \gamma^{t^{-1}}$.
Thus we have
\[[\gamma]\left(\left(s^{-1}\right)^{t} f^t s^{h^{-1} t} \right) =
\left[\gamma^{t^{-1}}\right]s^{-1}\cdot 
\left[\gamma^{t^{-1}}\right]f\cdot \left[\gamma^{t^{-1}}\right]s = 1_K,\]
and thus $\terr(w^a) = \supp(h^t)$. The claimed identity for
$\left[\gamma,w^a\right]\Yade$ follows from the
following calculation:
\begin{align*}
\left[\gamma,w^a\right]\Yade
&= \prod_{i=0}^{\vert h^t \vert - 1} [\gamma] 
\left(\left(s^{-1}\right)^{t} f^t s^{h^{-1} t}\right)^{(h^t)^{-i}} = 
\prod_{i=0}^{\vert h \vert - 1}
[\gamma] \left(s^{-1} f s^{h^{-1}}\right)^{t t^{-1} h^{-i} t} \\
&= \prod_{i=0}^{\vert h \vert - 1} 
\left[\gamma^{t^{-1}}\right] \left(\left(s^{-1}\right)^{h^{-i}} f^{h^{-i}} 
s^{h^{-i-1}}\right) = \left[\gamma^{t^{-1}}\right]s^{-1}\cdot\left[\gamma^{t^
{-1}}\!,\,w\right]\Yade \cdot \left[\gamma^{t^{-1}}\right]s.
\end{align*}
\end{proof}

One can now derive the following corollary relating the territories of
conjugate wreath cycles.
\begin{cor}\label{cor:conjugate_terr}
Let $w=(f,h) \in S$ be a wreath cycle and $a=(s,t)\in S$.
Then \[\terr(w^a)=\terr(w)^t.\]
\end{cor}
\begin{proof}
The claim follows from the proof of Lemma \ref{conjugateInWreath} if $h=1_H$.
If $h\neq 1_H$, the claim follows because $w$ and $w^a$ are wreath cycles
and hence
\[\terr(w^a) = \supp(h^t) = \supp(h)^t = \terr(w)^t.\qedhere\]
\end{proof}

Next we show the converse of Lemma \ref{conjugateInWreath}, i.e.\ that two wreath cycles with
the same load are always conjugate in $S$. In order to do so, we need the
following lemma, which is \cite[Theorem~2]{Ore} (we repeat the proof as we
require the constructed elements).
\begin{lemma}\label{le:conjugateCs}
Let \(K\) be a group, $\ell\in \mathbb{Z}_{>0}$,
\(a_0, a_1,\ldots, a_{\ell-1}, b_0,b_1,\ldots, b_{\ell-1}\in K\)
and set
\[a \coloneqq a_0a_1\dotsm a_{\ell-1}\text{ and } b \coloneqq b_
0b_1\dotsm b_{\ell-1}.\]
Then there exist \(c_0,\dotsc, c_\ell \in K\) where \(c_0
= c_\ell\) such that \(b_i = c_i^{-1}a_ic_{i+1}\) for every \(i \in \zset{\ell-1}\) 
if and only if \(a\) and \(b\) are conjugate in \(K\).
\end{lemma}
\begin{proof}
Let \(a\) and \(b\) be conjugate in \(K\) and \(c_0 \in
K\) such that \(a^{c_0} = b\), in particular
$(a_0a_1\cdots a_{\ell-1})^{c_0} = b_0b_1\cdots \!b_{\ell-1}.$
Define
$c_i \coloneqq a_{i-1}^{-1}a_{i-2}^{-1}\dotsm a_0^{-1}c_0b_0
b_1\dotsm b_{i-1} $,
then a short calculation shows that \(c_i\) has the desired property.
Conversely, if such \(c_0,\dotsc, c_\ell \in K\) exist, we have
\((a_0a_1\dotsm a_{\ell-1})^{c_0} = b_0b_1\dotsm b_{\ell-1}\) and thus 
\(a^{c_0} = b\).
\end{proof}
This enables us to show the following lemma, which is \cite[Theorem 6]{Ore}
restated in our notation.
\begin{lemma}\label{lem:load_invariant}

Let $w$, $v\in S$ be wreath cycles.
Then $w$ and $v$ are conjugate in $S$ if and only if they have
the same load. 
\end{lemma}
\begin{proof}
The if direction is Lemma \ref{conjugateInWreath}. 
Now consider the wreath cycles $w=(f,h)$, $v=(e,g) \in S$ with the same load 
$\left(k^K,j\right)$.
We construct an element $a=(s,t) \in S$ with $w^a=v$.\\
First assume $j=1$. Then $\vert\terr(w)\vert = \vert\terr(v)\vert = 1$ and
for some 
$\gamma_0 \in \terr(w)$ there exists a $t\in \symg$ with $\gamma_0^t \in
\terr(v)$. As the Yade-classes of $w$ and $v$ agree there exists a $c\in K$
with 
\[c^{-1}\cdot[\gamma_0]f\cdot c =  c^{-1}\cdot[\gamma_0, w]\Yade\cdot c =
[\gamma_0^t,v]\operatorname{Yade} = [\gamma_0^t]e.\]
Using
\[
s: \Gamma\to K, \gamma\mapsto \begin{cases}
c,&\text{ if } \gamma=\gamma_0\\
1_K,& \text{ else }
\end{cases}
\]
we conclude $w^{(s,t)} = v$.\\
Now assume $j>1$. As the order of $h$ equals that of $g$ there exists a $t \in \symg$
with $h^t=g$. We continue
by constructing the base component $s$ of a conjugating element. 
Fix $\gamma_0 \in \supp(h)$. As
\[[\gamma_0,w]\Yade = \prod \limits_{i = 0}^{j-1}[\gamma_0]f^{h^{-i}}
\text{ and } \left[\gamma_0^t, v\right]\Yade = \prod
\limits_{i=0}^{j-1} [\gamma_0^t]e^{g^{-i}}\in K\]
are conjugate in $K$, by Lemma \ref{le:conjugateCs} there exist
$c_0,\dotsc,c_j \in K$ where $c_0=c_j$ such that
\[c_i^{-1}\cdot [\gamma_0]f^{h^{-i}}\cdot c_{i+1}=
\left[\gamma_0^t\right]e^{g^{-i}}.\]
Now define 
\[s\colon \Gamma \to K, \gamma \mapsto \begin{cases}
	1_K,&\text{ if } \gamma \not \in \terr(w),\\
	c_i, &\text{ if } \gamma=\gamma_0^{h^i} \text{ for } 0\leq i < j.
\end{cases}\]
We now show that $a=(s,t)$ has the desired property. First note 
\[
\terr(v) = \supp(g) = \supp(h)^t = \terr(w)^t	
\]
and
\[
w^a = (f,h)^{(s,t)} = \left((s^{-1})^{t}\cdot f^t\cdot s^{h^{-1}\cdot t},
t^{-1}\cdot h\cdot t \right)	
= \left((s^{-1})^{t}\cdot f^t\cdot s^{h^{-1}\cdot t}, g \right).
\]
Next we prove that the base component of $w^a$ equals that of $v$. Let
$\gamma\in \Gamma$ and observe  
\begin{align}\label{equation-gamma}
[\gamma]\left((s^{-1})^{t}\cdot f^t\cdot s^{h^{-1} \cdot t}\right) =
\left[\gamma^{t^{-1}}\right]
s^{-1}\cdot \left[\gamma^{t^{-1}}\right]f
\cdot \left[\gamma^{t^{-1}\cdot h}\right]s
\end{align}
We now distinguish two cases for $\gamma \in \Gamma$:
\begin{itemize}
\item $\gamma \not\in \terr(v)$: Using Corollary \ref{cor:conjugate_terr} we obtain
$\terr(w^a) = \terr(w)^t = \terr(v)$ and hence \[[\gamma]\left((s^{-1})^{t}\cdot f^t
s^{h^{-1}\cdot t}\right) = 1_K = [\gamma]e.\]
\item $\gamma \in \terr(v)=\supp(g)$: Then $\gamma^{t^{-1}} \in
\supp(h)$ and there exists an $i\in \{0,\dotsc, j -1\}$
with $\gamma^{t^{-1}}=\gamma_0^{h^i}$.
By Equation \ref{equation-gamma}
\begin{align*}
[\gamma]\left((s^{-1})^{t}\cdot f^t \cdot s^{h^{-1} t}\right) &= \left[\gamma^{t^{-1}}\right]
s^{-1}\cdot \left[\gamma^{t^{-1}}\right]f
\cdot \left[\gamma^{t^{-1}\cdot h}\right]s 
= c_{i}^{-1}\cdot \left[\gamma_0^{h^i}\right]f \cdot c_{i+1}\\
&= \left[\gamma_0^{t}\right]e^{g^{-i}}
= \left[\gamma_0^{t\cdot g^{i}}\right]e=
\left[\gamma_0^{t\cdot (t^{-1}\cdot h\cdot t)^i}\right]e
= \left[\gamma_0^{h^i\cdot t}\right]e = [\gamma]e
\end{align*}
which concludes the proof.\qedhere
\end{itemize}
\end{proof}

We generalise the above results to arbitrary wreath product elements. For
this, we need to introduce a few additional concepts.
\begin{definition}\label{def:Cset}
Let $w=(f,h) = w_1\cdots w_\ell \in W$ be an arbitrary wreath product
element in disjoint wreath cycle decomposition, i.e. the $w_i\in S$ are
disjoint wreath cycles. Define $\C(w) \coloneqq
\{w_1,\dots, w_\ell\}$ as the set of all wreath cycles in a disjoint wreath
cycle decomposition of $w$ and
\[\L(w) \coloneqq \left\{\, \ld(z) \,:\, z \in \C(w)\, \right\}
\text{ and for } L\in \L(w) \, \text{ set }
\C(L,w) \coloneqq \left\{\, z\in \C(w) \,:\, \ld(z)=L \right\}.\]
Further, set $\C^{*}(w) \coloneqq \{z=(e,g)\in \C(w) \,:\, g\neq 1_H \}$ as the set of
all wreath cycles of $w$ with non-trivial top component.
\end{definition}
Note that the above sets are well-defined as the cycles in a wreath cycle
decomposition are unique up to permutation. Moreover, one can now write
\begin{align}\label{notation:wr_decomp}
w=w_1\cdots w_\ell = \prod_{L\in \L(w)}\prod_{z\in \C(L,w)}z.
\end{align}
The disjoint cycle decomposition of a permutation induces a partition
of the underlying set by considering the support of each cycle. We extend
this concept and decompose the territory of each wreath product element in
accordance with the decomposition given in Equation \ref{notation:wr_decomp}.
\begin{definition}\label{def:territory_decomposition}
Let $w\in W$ and define the \emph{territory decomposition} of $w$ as
\[\mathcal{P}(w) =
\bigtimes_{L\in \L(w)}
\{\,\terr(z) \,:\, z\in \C(L,w)\}.
\]
\end{definition}
\begin{remark}\label{rem:A_mat}
Suppose that $K$ has finitely many conjugacy classes $k_1^K,\dots, k_r^K$
and let $w\in W$. Suppose that $A= \bigtimes_{L\in \L(w)}A_L$ is a set indexed by
$\L(w)$. In our examples and in analogy to the example by Specht in \cite{Specht}
we
illustrate $A$ as an $r\times \vert \Gamma \vert$-matrix
 $M$ via $M_{i,j} = A_L$, if $L=(k_i^K,j)\in
L(w)$ and $M_{i,j} = \emptyset$ else.
\end{remark}

\begin{example}
\label{Example-Wreath-Product-Territory-Partition}
Let $K \coloneqq \sym(\{1, \dots, 4\})$,
$\Gamma \coloneqq \{1,\dots, 8\}$ and
$S \coloneqq K \wr \sym(\Gamma)$.
We give an example for $\mathcal{P}(w)$ by considering the wreath cycle
decomposition of the element $w = (f,h) \in S$ from Example
\ref{Example-Wreath-Product-Wreath-Cycle-Decomp}.
For this we choose \[\mathcal{R}(K) \coloneqq
\{\,k_1 \coloneqq (),\, k_2 \coloneqq (1,2),\,k_3 \coloneqq (1,2)(3,4),\, k_4
\coloneqq (1,2,3),\, k_5 \coloneqq (1,2,3,4)\,\}\] as a set of representatives 
for the conjugacy classes of
$K$ and set $r \coloneqq \vert \mathcal{R}(K) \vert = 5$.
Recall the wreath cycle decomposition of $w$ given by the following four wreath cycles:
\[
\arraycolsep=0.5em\def\arraystretch{2.2}
\begin{array}{l*{10}{c}}
\highlightB{w_1} = \bigl(
& \overset{\highlightB{1}}{(1,2)(3,4)},
& \overset{\highlightB{2}}{(3,4)},
& \overset{\highlightZ{3}}{()},
& \overset{\highlightZ{4}}{()},
& \overset{\highlightZ{5}}{()},
& \overset{\highlightZ{6}}{()},
& \overset{\highlightZ{7}}{()},
& \overset{\highlightZ{8}}{()};
& \overset{\text{top}}{(\highlightB{1,2})}
& \bigr), \\
\highlightC{w_2} = \bigl(
& \overset{\highlightZ{1}}{()},
& \overset{\highlightZ{2}}{()},
& \overset{\highlightC{3}}{()},
& \overset{\highlightC{4}}{(1,2)},
& \overset{\highlightZ{5}}{()},
& \overset{\highlightZ{6}}{()},
& \overset{\highlightZ{7}}{()},
& \overset{\highlightZ{8}}{()};
& \overset{\text{top}}{(\highlightC{3,4})}
& \bigr), \\
\highlightD{w_3} = \bigl(
& \overset{\highlightZ{1}}{()},
& \overset{\highlightZ{2}}{()},
& \overset{\highlightZ{3}}{()},
& \overset{\highlightZ{4}}{()},
& \overset{\highlightD{5}}{(1,2,3)},
& \overset{\highlightD{6}}{()},
& \overset{\highlightZ{7}}{()},
& \overset{\highlightZ{8}}{()};
& \overset{\text{top}}{(\highlightD{5,6})}
& \bigr), \\
\highlightA{w_4} = \bigl(
& \overset{\highlightZ{1}}{()},
& \overset{\highlightZ{2}}{()},
& \overset{\highlightZ{3}}{()},
& \overset{\highlightZ{4}}{()},
& \overset{\highlightZ{5}}{()},
& \overset{\highlightZ{6}}{()},
& \overset{\highlightA{7}}{(1,2)},
& \overset{\highlightZ{8}}{()};
& \overset{\text{top}}{()}
& \bigr).
\end{array}
\]
The Yade classes for each of the $w_i$ are
\begin{align*}
[\highlightB{1}, \highlightB{w_1}]\Yade
&= (1,2)(3,4) \cdot (3,4) = (1,2) \in k_2^K, \\
[\highlightC{3}, \highlightC{w_2}]\Yade
&= () \cdot (1,2) = (1,2) \in k_2^K, \\
[\highlightD{5}, \highlightD{w_3}]\Yade
&= (1,2,3) \cdot () = (1,2,3) \in k_4^K, \\
[\highlightA{7}, \highlightA{w_4}]\Yade
&= (1,2) \in k_2^K.
\end{align*}
We write $\mathcal{P}(w)$ in matrix notation, where according to Remark
\ref{rem:A_mat} the entry in position $(i,j)$ is $\bigcup_{z\in
\C(w)}\{terr(z) \, : \, \ld(z)=(k_i^K,j)\}$. 
We omit the additional set braces for each entry
of $\P(w)$ and write $\zero$ if an entry of $\P(w)$ is the empty set, so 
\[
\mathcal{P}(w) =
\begin{blockarray}{*{9}{c}}
\scriptstyle \highlightZ{1} & \scriptstyle \highlightZ{2} & \scriptstyle \highlightZ{3} & \scriptstyle \highlightZ{4} & \scriptstyle \highlightZ{5} & \scriptstyle \highlightZ{6} & \scriptstyle \highlightZ{7} & \scriptstyle \highlightZ{8} & \\
\begin{block}{[*{8}{c}]>{\scriptstyle}r}
\zero & \zero & \zero & \zero & \zero & \zero & \zero & \zero & \highlightZ{k_1}\\
\{\highlightA{7}\} & \{\highlightB{1,2}\}, \{\highlightC{3,4}\} & \zero & \zero & \zero & \zero & \zero & \zero & \highlightZ{k_2} \\
\zero & \zero & \zero & \zero & \zero & \zero & \zero & \zero & \highlightZ{k_3} \\
\zero & \{\highlightD{5,6}\} & \zero & \zero & \zero & \zero & \zero & \zero & \highlightZ{k_4} \\
\zero & \zero & \zero & \zero & \zero & \zero & \zero & \zero & \highlightZ{k_5} \\
\end{block}
\end{blockarray}.
	\]
\end{example}

\begin{cor}\label{conjOfTerrDecomp}
Let $w=(f,h) \in W$ and $a=(s,t) \in W$. Then $\mathcal{P}(w ^ a) =
\mathcal{P}(w) ^ t$ and in
particular $\terr(w^a)=\terr(w)^t$. 
\end{cor}
\begin{proof}
By Lemma \ref{conjugateInWreath} we have $\ld(z) = \ld(z^a)$ for all $z\in
\C(w)$ and thus $\C(L,w^a) = \C(L,w)^a$ for all $L\in \L(w)$. Therefore,
\begin{align*}
\P(w^a) &= \bigtimes_{L\in \L(w^a)} \{\terr(z) \,:\, z\in \C(L,w^a)\} = 
\bigtimes_{L\in \L(w)} \{\terr(z^a) \,:\, z\in \C(L,w)\}\\
&=\bigtimes_{L\in \L(w)} \{\terr(z)^t \,:\, z\in \C(L,w)\} = \P(w)^t
\end{align*}
by Corollary \ref{cor:conjugate_terr}.
\end{proof}

The following theorem, which is an explicit version of Theorem
\ref{main-thm-1}, gives us a way to test whether two elements of the wreath product $S$
are conjugate in  $W$ after having computed their wreath
cycle decomposition. The proof of this theorem is constructive as it shows
how to construct a conjugating element if it exists and therefore solves
the conjugacy problem.

\begin{theorem}\label{conjugationInWreathProd}
Two elements $w=(f,h),\, v=(e,g)\in S$ are conjugate in $W$ if and
only if there exists a $t\in H$ such that  $h^t=g$ and $\mathcal{P}(w)^t=\mathcal{P}(v)$.
\end{theorem}
\begin{proof}
We first show the if direction: If there exists $a =(s,t) \in W$ such that
$w^a = v$, then $h^t=g$ and by Corollary \ref{conjOfTerrDecomp} the claim follows.

We now show the only-if direction, so assume the existence of a $t\in H$
with $\mathcal{P}(w)^t=\mathcal{P}(v)$ and $h^t=g$, hence $\L(w)=\L(v)$.
As we can identify each cycle with its territory, $t$ induces
a unique bijection $\sigma: \C(w)\xrightarrow{1:1}\C(v)$ such that
$\sigma$ maps $\C(L,w)$ to $\C(L,v)$ bijectively for all $L\in \L(w)$ and 
$\terr(z)^t = \terr([z]\sigma)$ for all $z\in \C(L,w)$. Moreover, $t$
conjugates the
top component of $z\in \C(L,w)$ to the top component of
$[z]\sigma$ as $h^t=g$. Additionally, $\ld(z) = \ld([z]\sigma)$ for
all $z\in \C(L,w)$.\\
We now construct a base component $s \in K^\Gamma$ such that $a=(s,t)\in W$ satisfies $w^a=v$.
As in the proof of Lemma
\ref{lem:load_invariant}, for each $z\in \C(L,w)$, we can find an $s_z \in
K^\Gamma$ with $\left(s_z,t\right)^{-1}\cdot z\cdot
\left(s_z,t\right)=[z]\sigma$. Define 
\[
s \coloneqq \prod_{z\in \C(w)}s_z
\]
as the product of all these maps for all disjoint wreath cycles in a
decomposition of $w$. By construction
$[\gamma]s=[\gamma]s_z$ if $\gamma \in \terr(z)$ since
$[\gamma]s_{y}=1_K$ for all $y \neq z$.
Thus
\[
w^a =\prod_{z\in \C(w)}z^{(s,t)} = \prod_{z\in \C(w)}(s_z,t)^{-1}\cdot z\cdot (s_z,t) = 
\prod_{z\in \C(w)}[z]\sigma = v,
\]
which concludes the proof.
\end{proof}
There are several ways of finding a conjugating element. One could 
formulate this as a single backtrack problem seeking 
an element $t\in H$ which simultaneously conjugates $h$ to $g$ and maps
$\mathcal{P}(w)$ to $\mathcal{P}(v)$. 
Alternatively, one could first find an element
$t\in H$ with $h=g^t$, then compute $C_H(h)$ and check if $\mathcal{P}(v)^{t}$ is in
the orbit of $\mathcal{P}(w)$ under $C_H(h)$. If not, $w$ and $v$ are not
conjugate. Backtrack may also be helpful in the second approach. The
implementation by the third author in \cite{Rober_Package} uses the second
approach.
For many search problems in permutation groups, partition backtrack is
still the state of the art algorithm. For an exposition on the backtrack
strategy frequently used, namely partition
backtrack, see for example \cite{Leon}.

The special case of $H=\symg$ is already described in \cite[Theorem
4.2.8]{Kerber}. We obtain the same result as a corollary to the above
theorem.
\begin{cor}\label{cor:conjinS}
Two elements $w,\, v \in S$ are conjugate in $S$ if and only if
$\L(w)=\L(v)$ and $\vert \C(L,w) \vert = \vert \C(L,v)\vert $ for all $L\in
\L(w)$.
\end{cor}
Applying the above theory, we give an example for conjugacy testing in wreath products.
\begin{example}
\label{Example-Wreath-Product-Conjugacy-Problem}
We use the notation from Example \ref{Example-Wreath-Product-Territory-Partition}
and highlight wreath cycle decompositions by colouring points of $\Gamma$.
We computed a wreath cycle decomposition of $w$ in
Example \ref{Example-Wreath-Product-Wreath-Cycle-Decomp} and $\mathcal{P}(w)$ in 
Example \ref{Example-Wreath-Product-Territory-Partition} using the same colours. 
Consider the elements $w = (f,h),\, v = (e,g) \in S$:
\[
\arraycolsep=0.15em\def\arraystretch{2.2}
\begin{array}{l*{10}{c}}
w \coloneqq \bigl(
& \overset{\highlightB{1}}{(1,2)(3,4)},
& \overset{\highlightB{2}}{(3,4)},
& \overset{\highlightC{3}}{()},
& \overset{\highlightC{4}}{(1,2)},
& \overset{\highlightD{5}}{(1,2,3)},
& \overset{\highlightD{6}}{()},
& \overset{\highlightA{7}}{(1,2)},
& \overset{\highlightZ{8}}{()};
& \overset{\text{top}}{(\highlightB{1,2})(\highlightC{3,4})(\highlightD{5,6})}
&\bigr), \\[10pt]
v \coloneqq \bigl(
& \overset{\highlightF{1}}{(3,4)},
& \overset{\highlightF{2}}{()},
& \overset{\highlightH{3}}{()},
& \overset{\highlightH{4}}{(1,2,3)},
& \overset{\highlightG{5}}{(1,2)},
& \overset{\highlightG{6}}{()},
& \overset{\highlightZ{7}}{()},
& \overset{\highlightE{8}}{(3,4)};
& \overset{\text{top}}{(\highlightF{1,2})(\highlightH{3,4})(\highlightG{5,6})}
&\bigr).
\end{array}
\]
The top components of $w$ and $v$ are equal and we have
\[
\begin{blockarray}{*{10}{c}}
& \scriptstyle \highlightZ{1} & \scriptstyle \highlightZ{2} & \scriptstyle \highlightZ{3} & \scriptstyle \highlightZ{4} & \scriptstyle \highlightZ{5} & \scriptstyle \highlightZ{6} & \scriptstyle \highlightZ{7} & \scriptstyle \highlightZ{8} & \\
\begin{block}{l[*{8}{c}]>{\scriptstyle}r}
& \zero & \zero & \zero & \zero & \zero & \zero & \zero & \zero & \highlightZ{k_1}\\
& \{\highlightA{7}\} & \{\highlightB{1,2}\}, \{\highlightC{3,4}\} & \zero & \zero & \zero & \zero & \zero & \zero & \highlightZ{k_2} \\
\mathcal{P}(w) = & \zero & \zero & \zero & \zero & \zero & \zero & \zero & \zero & \highlightZ{k_3} \\
& \zero & \{\highlightD{5,6}\} & \zero & \zero & \zero & \zero & \zero & \zero & \highlightZ{k_4} \\
& \zero & \zero & \zero & \zero & \zero & \zero & \zero & \zero & \highlightZ{k_5} \\
\end{block} \\
& \scriptstyle \highlightZ{1} & \scriptstyle \highlightZ{2} & \scriptstyle \highlightZ{3} & \scriptstyle \highlightZ{4} & \scriptstyle \highlightZ{5} & \scriptstyle \highlightZ{6} & \scriptstyle \highlightZ{7} & \scriptstyle \highlightZ{8} & \\
\begin{block}{l[*{8}{c}]>{\scriptstyle}r}
& \zero & \zero & \zero & \zero & \zero & \zero & \zero & \zero & \highlightZ{k_1}\\
& \{\highlightE{8}\} & \{\highlightF{1,2}\}, \{\highlightG{5,6}\} & \zero & \zero & \zero & \zero & \zero & \zero & \highlightZ{k_2} \\
\mathcal{P}(v) = & \zero & \zero & \zero & \zero & \zero & \zero & \zero & \zero & \highlightZ{k_3} \\
& \zero & \{\highlightH{3,4}\} & \zero & \zero & \zero & \zero & \zero & \zero & \highlightZ{k_4} \\
& \zero & \zero & \zero & \zero & \zero & \zero & \zero & \zero & \highlightZ{k_5} \\
\end{block}
\end{blockarray}.
\]
Observe $\L(v)=\L(w)$ and $\vert \C(L,v)\vert=\vert \C(L,w)\vert$ for all
$L\in \L(v)$ and thus $v$ and $w$ are conjugate in $S = K \wr \sym(\Gamma)$
by Corollary \ref{cor:conjinS}.\\
Now let us consider three different choices for the  top group $H_i$
($1\leq i\leq 3$) for the 
wreath product $W_i \coloneqq K \wr H_i$ and decide whether $w$ and $v$ are
conjugate in $W$.
For this we need to check if $\mathcal{P}(v) \in \mathcal{P}(w)^{C_H(h)}$.
We use three different top groups:
\[
\begin{array}{lcl}
H_1 \coloneqq& \langle ~ (1,2)(3,4),~ (1,2,3,4),~ (5,6),~ (7,8) ~ \rangle & \cong D_8 \times C_2 \times C_2, \\
H_2 \coloneqq& \langle ~ (1,2)(3,4)(5,6),~ (3,5)(4,6)(7,8) ~ \rangle & \cong C_2 \times C_2, \\
H_3 \coloneqq& \langle ~ (1,2)(3,4)(5,6),~ (7,8) ~ \rangle & \cong C_2 \times C_2.
\end{array}
\]
Since the submatrix induced by the columns from $3$ to $8$ in $\mathcal{P}(w)$
consists of empty entries, we condense the notation and write $\zero$ for a
matrix of appropriate dimension with empty entries.
We have
\begin{align*}
\mathcal{P}(w)^{C_{H_1}(h)} &= \left\{
\begin{blockarray}{*{4}{c}}
\scriptstyle \highlightZ{1} & \scriptstyle \highlightZ{2} & \scriptstyle \highlightZ{3, \dots, 8} & \\
\begin{block}{[*{3}{c}]l}
\zero & \zero & \zero & \\
\{7\} & \{1,2\}, \{3,4\} & \zero & \\
\zero & \zero & \zero & \\
\zero & \{5,6\} & \zero \\
\zero & \zero & \zero & , \\
\end{block}
\end{blockarray}
\begin{blockarray}{*{3}{c}}
\scriptstyle \highlightZ{1} & \scriptstyle \highlightZ{2} & \scriptstyle \highlightZ{3, \dots, 8} \\
\begin{block}{[*{3}{c}]}
\zero & \zero & \zero \\
\{8\} & \{1,2\}, \{3,4\} & \zero \\
\zero & \zero & \zero \\
\zero & \{5,6\} & \zero \\
\zero & \zero & \zero \\
\end{block}
\end{blockarray}
\right\}, \\
\mathcal{P}(w)^{C_{H_2}(h)}  &= \left\{
\begin{blockarray}{*{4}{c}}
\scriptstyle \highlightZ{1} & \scriptstyle \highlightZ{2} & \scriptstyle \highlightZ{3, \dots, 8} & \\
\begin{block}{[*{3}{c}]l}
\zero & \zero & \zero & \\
\{7\} & \{1,2\}, \{3,4\} & \zero & \\
\zero & \zero & \zero & \\
\zero & \{5,6\} & \zero \\
\zero & \zero & \zero & , \\
\end{block}
\end{blockarray}
\begin{blockarray}{*{3}{c}}
\scriptstyle \highlightZ{1} & \scriptstyle \highlightZ{2} & \scriptstyle \highlightZ{3, \dots, 8} \\
\begin{block}{[*{3}{c}]}
\zero & \zero & \zero \\
\{8\} & \{1,2\}, \{5,6\} & \zero \\
\zero & \zero & \zero \\
\zero & \{3,4\} & \zero \\
\zero & \zero & \zero \\
\end{block}
\end{blockarray}
\right\},\\
\mathcal{P}(w)^{C_{H_3}(h)} &= \mathcal{P}(w)^{C_{H_1}(h)}.
\end{align*}
Hence $w$ and $v$ are not conjugate in $W_1$ and $W_3$, but are conjugate in $W_2$.
Now let us construct an element $a = (s, t) \in W_2$ with $w^a = v$.
First we compute an element $t\in C_{H_2}(h)$ with $\mathcal{P}(w)^t =
\mathcal{P}(v)$, for example 
$t = (3,5)(4,6)(7,8) \in C_{H_2}(h)$.
Using the above colouring to encode the wreath cycles, we write $w$ and $v$ in a
disjoint wreath cycle decomposition as in Equation
\ref{notation:wr_decomp}, where $w_{i,j,\ell}$ denotes the $\ell$-th wreath cycle
of load $(k_i^K,j)$:
\begin{align*}
w &= \highlightA{w_{2,1,1}}
\cdot \highlightB{w_{2,2,1}}
\cdot \highlightC{w_{2,2,2}}
\cdot \highlightD{w_{4,2,1}}, \\
v &= \highlightE{v_{2,1,1}}
\cdot \highlightF{v_{2,2,1}}
\cdot \highlightG{v_{2,2,2}}
\cdot \highlightH{v_{4,2,1}}.
\end{align*}
Suppose $w_{i,j,\ell} =(f_{i,j,\ell}, h_{i,j,\ell})$ and $v_{i,j,\ell} =
(e_{i,j,\ell}, h_{i,j,\ell})$.
Next, compute the bijection $\sigma\colon \C(w)\xrightarrow{1:1}\C(v)$ recording
the mapping induced by $a$, i.e.
$z^a = [z]\sigma$ for every $z \in \C(w)$, where the base component of $a = (s,t)$
is yet to be constructed.
As $\sigma$ only depends on the top component of $a$, it is already
determined by $[w_{i,j,\ell}]\sigma = v_{i,j,\ell}$.
Now it remains to construct the base component $s \in K^\Gamma$ of the
conjugating element $a$.
For this we construct elements $s_{i,j,\ell} \in K^\Gamma$,
such that $w_{i,j,\ell}^{(s_{i,j,\ell},\, t)} = v_{i,j,\ell}$ as in the
proof of Theorem \ref{conjugationInWreathProd}.
We demonstrate this for the wreath cycle $\highlightB{w_{2,2,1}}$ which
takes the place of the cycle $z$ in the proof.
First we compute two Yades:
\begin{align*}
[1, \highlightB{w_{2,2,1}}]\Yade
&= (1,2)(3,4) \cdot (3,4) = (1,2), \\
\left[1^t, \highlightF{v_{2,2,1}}\right]\operatorname{Yade}
&= (3,4) \cdot () = (3,4).
\end{align*}
Next, note that $x \coloneqq (1,3)(2,4) \in K$ conjugates $(1,2)$ to $(3,4)$.
Then compute the following elements of $K$ used in Lemma \ref{le:conjugateCs}:
\begin{align*}
c_{0} &\coloneqq x
= (1,3)(2,4)\\
c_{1} &\coloneqq [1]\highlightB{f_{2,2,1}}^{-1} \cdot c_{0} \cdot \left[1^t
\right]\highlightF{e_{2,2,1}}
= \bigl((1,2)(3,4)\bigr)^{-1} \cdot (1,3)(2,4) \cdot (3,4) = (1,3,2,4).
\end{align*}
We proceed to define the element $\highlightB{s_{2,2,1}}$ as
\[
\highlightB{s_{2,2,1}} \colon \Gamma \to K, \gamma \mapsto \begin{cases}
1_K, &\text{ if } \gamma \not \in \terr(\highlightB{w_{2,2,1}}),\\
c_{i}, &\text{ if } \gamma=1^{\highlightB{h_{2,2,1}^i}} \text{ for } 0\leq i
\leq 1.
\end{cases}\]
Hence we have
\[
\begin{array}{l*{9}{c}}
\highlightB{s_{2,2,1}} = \bigl(
& \overset{\highlightB{1}}{(1,3)(2,4)},
& \overset{\highlightB{2}}{(1,3,2,4)},
& \overset{\highlightZ{3}}{()},
& \overset{\highlightZ{4}}{()},
& \overset{\highlightZ{5}}{()},
& \overset{\highlightZ{6}}{()},
& \overset{\highlightZ{7}}{()},
& \overset{\highlightZ{8}}{()}
&\bigr).
\end{array}
\]
Analogously we compute $\highlightA{s_{2,1,1}}$, $\highlightC{s_{2,2,1}}$ and $\highlightD{s_{4,2,1}}$.
This yields
\[
\arraycolsep=0.18em\def\arraystretch{2.2}
\begin{array}{l*{10}{c}}
a = \bigl(
& \overset{\highlightB{1}}{(1,3)(2,4)},
& \overset{\highlightB{2}}{(1,3,2,4)},
& \overset{\highlightC{3}}{()},
& \overset{\highlightC{4}}{(1,2)},
& \overset{\highlightD{5}}{()},
& \overset{\highlightD{6}}{(1,3,2)},
& \overset{\highlightA{7}}{(1,3)(2,4)},
& \overset{\highlightZ{8}}{()};
& \overset{\text{top}}{(3,5)(4,6)(7,8)}
&\bigr)\,.
\end{array}
\]
\end{example}

\section{Conjugacy classes in wreath products}
Recall the setting from Hypothesis \ref{setting}.
In this section, we
parameterise the $W$-conjugacy classes $w^W$
of arbitrary elements $w\in S$. This is achieved via defining bijections
between certain iterated  
cartesian products and $w^W$.
The notation is chosen to reflect the way we construct these conjugacy
classes. When using a cartesian product $A\times B$ to parameterise a set, 
one should view this as first choosing an element in $A$ and then in
$B$. 

The conjugacy class sizes and the number of conjugacy classes in the
full monomial group $S$ are already known, see for example James and Kerber
\cite[4.2.9, 4.2.10]{Kerber}.\\
The maps we define to parameterise the $W$-conjugacy classes of elements of
$S$ are constructed in such a
way that they can be implemented directly in computer algebra systems such
as \textsf{GAP}\cite{gap} or \textsc{Magma}\cite{Magma}.
For instance, the third author implemented
computation of $W$-conjugacy classes in this way
in the \textsf{GAP} package \textit{WPE}, see \cite{Rober_Package}.\\
We extend the notation $\C(W)$ from Definition \ref{def:yade} naturally to
refer to wreath cycles of $W$ with a fixed top component $h$.
\begin{definition}
For \(h \in H\) denote the set of all wreath cycles of \(W\) with top component \(h\) by
\[\C(W,h) \coloneqq \{(f,h) \in W \,:\, (f,h) \text{ is a wreath cycle}\}.\]
\end{definition}
Let $1_H \neq h\in H$ be a single cycle and
$\gamma_0\in\supp(h)$. Given an $x\in K$, we 
construct all wreath cycles $w=(f,h)\in K\wrg H$ with top component $h$ and
$[\gamma_0, w]\Yade =x$. The image of the map $[h,\gamma_0,x,-]\mathcal{B}$
defined in the following lemma yields the base components of the desired
wreath cycles.
\begin{lemma}
\label{wreath-cycle-fixed-yade-parameterisation}
Let $1_H \neq h \in H$ be a cycle, $\gamma_0 \in \supp(h)$
and $x\in K$. Define
$$[h,\gamma_0,x,-]\mathcal{B} \colon K^{\supp(h)\setminus\{\gamma_0\}}
\hookrightarrow K^\Gamma,\,
d \mapsto [h,\gamma_0,x,d]\mathcal{B}$$ where 
\[
[h,\gamma_0,x,d]\mathcal{B}=
	\gamma \mapsto \begin{cases}
		x\cdot \prod\limits_{i=1}^{\vert h\vert -1} \left[\gamma^{h^{\vert h\vert -i}}\right]d^{-1}, & \text{if } \gamma = \gamma_0\\
		[\gamma]d,& \text{if } \gamma \in \supp(h)\setminus\{\gamma_0\}\\
		1_K, & \text{if } \gamma \in \Gamma \setminus \supp(h).
	\end{cases}
\]
Then the following statements hold:
\begin{enumerate}
\item $[h,\gamma_0,x,-]\mathcal{B}$ is an injection.
\item $w\coloneqq ([h,\gamma_0,x,d]\mathcal{B},h) \in \C(W,h)$, i.e. $w$ is
a wreath cycle of $W$ with top component $h$.
\item The map $[h,\gamma_0,x,-]\mathcal{B}$ induces a bijection
\[
\varphi\colon K^{\supp(h)\setminus\{\gamma_0\}} \xrightarrow{1:1} \{w \in \C(W,h) \,:\,
[\gamma_0, w]\Yade=x\},\, d \mapsto ([h,\gamma_0,x,d]\mathcal{B}, h).
\]
\item If $K$ is finite, we have \[\vert\{w \in \C(W,h) \,:\,
[\gamma_0, w]\Yade=x\}\vert = \vert K\vert^{\vert h\vert -1}.\]
\end{enumerate}
\end{lemma}

\begin{proof}
First note that $[h,\gamma_0,x,-]\mathcal{B}$ is injective since we embed
$d$ in $\supp(h)\setminus \{\gamma_0\}$.
Next we show that $\varphi$ is well-defined. 
Let $d \in K^{\supp(h)\setminus\{\gamma_0\}}$, define $f \coloneqq
[h,\gamma_0,x,d]\mathcal{B}$ and $w \coloneqq (f, h)$.
Then $\terr(w) = \supp(h)$ and
\[
[\gamma_0, w]\Yade = \prod_{i=0}^{\vert h\vert -1} \left[\gamma_0^{h^i}\right]f
= x \,\cdot \, \prod\limits_{i=1}^{\vert h\vert -1} \left[\gamma_0^{h^{\vert h\vert -i}}\right]d^{-1}
\, \cdot \, \prod_{i=1}^{\vert h\vert -1}\left[\gamma_0^{h^{i}}\right]d=x.
\]
This shows $\im(\varphi)\subseteq \{w \in \C(W,h) \,:\,
[\gamma_0, w]\Yade=x\}$ and clearly $\varphi$ is injective.
Now let $w=(f,h)\in \{w \in \C(W,h) \,:\,
[\gamma_0, w]\Yade=x\}$.
We construct an element $d\in K^{\supp(h)\setminus
\{\gamma_0\}}$ with $[h,\gamma_0,x,d]\mathcal{B} = f$. Define 
$d:\supp(h)\setminus\{\gamma_0\} \to K, \gamma\mapsto [\gamma]f$ and 
$e\coloneqq [h,\gamma_0,x,d]\mathcal{B}$. For any $\gamma\in
\Gamma\setminus\{\gamma_0\}$ we have $[\gamma]e = [\gamma]f$. Now suppose
$\gamma=\gamma_0$. Then 
\[
[\gamma_0]e = [\gamma_0, w]\Yade \,\cdot \, \prod
\limits_{i=1}^{\vert h\vert -1} \left[\gamma_0^{h^{\vert h\vert -i}}\right]f^{-1} =
[\gamma_0]f.
\]
In particular $e=f$ and thus part 3 follows immediately.
\end{proof}
\begin{example}
\label{Example-Yade-Mapping-B}
Let $K \coloneqq \sym(\{1, \dots, 4\})$,
$\Gamma \coloneqq \{1,\dots, 8\}$ and
$S \coloneqq K \wr \sym(\Gamma)$.
We compute the image of $d$ under $[h, \gamma_0, x, -]\mathcal{B}$, where
\vspace*{-15pt}	
\[
\begin{array}{cc}
& \scriptstyle \highlightZ{4} \\
h \coloneqq \highlightF{(1,4)},
\quad \gamma_0 \coloneqq 1,
\quad x \coloneqq (3,4)
\quad \text{ and } \quad d \coloneqq & \bigl( (1,2,3,4) \bigr).
\end{array}
\]
Let $e \coloneqq [h, \gamma_0, x, d]\mathcal{B}$.
For all points $\gamma \in \Gamma
\setminus \supp(h) = \{2,3,5,6,7,8\}$, we obtain $[\gamma]e = 1_K$.
Next we compute the images of the points of $\supp(h) = \{1, 4\}$ under
$e$: $[1]e = x \cdot \left[1 ^ h\right]d^{-1} =
(3,4) \cdot (1,4,3,2) = (1,4,2)$ and $[4]e = [4]d = (1,2,3,4)$.
Then
\[
\begin{array}{r*{9}{c}}
[h, \gamma_0, x, d]\mathcal{B}
= \bigl(
& \overset{\highlightF{1}}{(1,4,2)},
& \overset{\highlightZ{2}}{()},
& \overset{\highlightZ{3}}{()},
& \overset{\highlightF{4}}{(1,2,3,4)},
& \overset{\highlightZ{5}}{()},
& \overset{\highlightZ{6}}{()},
& \overset{\highlightZ{7}}{()},
& \overset{\highlightZ{8}}{()}
&\bigr).
\end{array}
\]
In particular, 
\[\left([h, \gamma_0, x, d]\mathcal{B},h\right) =
((1,4,2),\,(),\,(),\,(1,2,3,4),\,(),\,(),\,(),\,()\,; \, (1,4))\]
is contained in $\{\,w\in \C(W,h) \,:\,
[\gamma_0,w]\Yade=x\}$.
\end{example}

Using the above lemma we can derive the proportion of wreath cycles with
non-trivial top-component
whose Yade in a given point lies in a given subset $P\subseteq K$.

\begin{cor}\label{probshiftcor}
Let $K$ be finite, $1\neq h\in H$ a single cycle, $\gamma \in \supp(h)$ and $P\subseteq K$.
Then \[\frac{\vert\{w\in \C(W,h) \,:\, 
[\gamma,w]\Yade \in P\}\vert}{\vert \C(W,h)\vert}=
\frac{\vert P\vert}{\vert K\vert}.\]
\end{cor}

\begin{proof}
\begin{align*}\vert\{w\in \C(W,h) \,:\, [\gamma,w]\Yade \in P\}\vert
&=\big\vert\bigcupdot_{x\in P} \{w \in \C(W,h) \,:\, 
[\gamma,w]\Yade=x\}\big\vert\\
&=\sum_{x\in P}\vert \{w \in \C(W,h) \,:\, [\gamma,w]\Yade=x\}\vert
=\vert K\vert^{\vert h\vert-1}\vert P\vert.
\end{align*}
The result follows as $\vert \C(W,h) \vert =\vert K\vert^{\vert h\vert}$.
\end{proof}
We now turn towards parameterising conjugacy classes of arbitrary wreath
product elements.\\
We commence our investigation with the $W$-conjugacy class of a single
wreath cycle in $S$. Note that if the top component of a wreath cycle is the
identity we require an embedding of $\Gamma\times K$ into the base
component, which we achieve via the map $\mathcal{E}$ defined in the
following lemma, replacing the map $\mathcal{B}$
from Lemma \ref{wreath-cycle-fixed-yade-parameterisation}.
\begin{lemma}\label{lemma-conjugacyclass-wreathcycle}
Let $w=(f,h)\in S$ be a wreath cycle and $\gamma_0 \in \terr(w)$. Then, for
$h\neq 1_H$, the map
\[
h^H \times \left([\gamma_0, w]\Yade \right)^K \times K^{\terr(w)\setminus \{\gamma_0\}}
\xrightarrow{1:1}
w^W, (h^t, x, d) \mapsto \left(\left[h^t,\gamma_0^t, x,d^t\right]\mathcal{B}, h^t\right)
\]
and for $h=1_H$, the map
\[
\gamma_0^H \times \left([\gamma_0, w]\Yade\right)^K \xrightarrow{1:1}
w^W, (\gamma_0^t, x) \mapsto
([\gamma_0^t, x]\mathcal{E},1_H)
\]
are bijections into the conjugacy class of $w$ in $W$, where $\mathcal{B}$
is as in Lemma \ref{wreath-cycle-fixed-yade-parameterisation} and 
\[
\mathcal{E}\colon \Gamma\times K \hookrightarrow K^\Gamma, (\gamma', x)\mapsto
[\gamma',x]\mathcal{E}
\]
with \[[\gamma',x]\mathcal{E}\colon \Gamma\to K, \gamma \mapsto
\begin{cases}
x, &\text{ if } \gamma=\gamma'\\
1_K, &\text{ if } \gamma\neq \gamma'
\end{cases}\]
is an injection.
\end{lemma}
\begin{proof}
Let $v=(e,g)\in S$. Using Theorem \ref{conjugationInWreathProd},
$w$ is conjugate to $v$ in $W$ if and only if $v$ is a wreath cycle,
there exists a $t\in H$ with $h^t=g$ and $\P(w)^t = \P(v)$. As $w$ and $v$
are wreath cycles and the load of a wreath cycle is invariant under
conjugation we know that $\P(w)^t=\P(v)$ holds if and only if  
$\left([\gamma_0^t,v]\operatorname{Yade}\right)^K = \left([\gamma_0,
w]\Yade\right)^K$.

We now parameterise the  different elements $v\in S$ one can construct with these properties.
Suppose first $h\neq 1_H$. Then we need to choose an element $g=h^t\in h^H$ as a possible top
component for $v$ and an element $x \in \left([\gamma_0, w]\Yade\right)^K$
as the Yade of $v$ at the point $\gamma_0^t$.
By Lemma \ref{wreath-cycle-fixed-yade-parameterisation}, the possible
elements $v$ with the above requirements are parameterised by the following
bijection 
\[
K^{\terr(w)\setminus \{\gamma_0\}} \xrightarrow{1:1} \{v \in \C(S,g) \,:\,
[\gamma_0^t,v]\operatorname{Yade} = x\},\, d \mapsto \left(\left[h^t,\gamma_0^t, x,d^t\right]
\mathcal{B}, h^t\right),
\]
observing
$K^{\terr(w)\setminus\{\gamma_0\}}\xrightarrow{1:1}
K^{\supp\left(h^t\right)\setminus\left\{\gamma_0^t\right\}},\, d\mapsto d^t$ is a bijection.\\
Now let $h=1_H$. Then the top component of $v$ must equal $1_H$. Since $v$
is a wreath cycle, we must have $\vert \terr(v) \vert = 1$. For the
territory of $v$, we need to choose $\gamma_0^t \in \gamma_0^H$ and $x = [\gamma_0^t]e =
[\gamma_0^t,v]\operatorname{Yade} \in \left( [\gamma_0, w]\Yade\right)^K$. These choices fix the element $v$, which must equal $([\gamma_0^t, x]\mathcal{E}, 1_H)$.
\end{proof}
We now parameterise the $W$-conjugacy class of arbitrary wreath product
elements $w\in S$. Recall the definition of $\C(w)$ and $\C^{*}(w)$ from
Definition \ref{def:Cset}.
\begin{theorem}\label{thm-cc}
Let $w=(f,h)\in S$ be an arbitrary wreath
product element and for each $z\in \C(w)$ 
choose $\gamma_z \in \terr(z)$.
Fix a transversal $\{ t_1,\dots, t_m \}$ of the right cosets of $C_H(h)$ in $H$.\\
Then the $W$-conjugacy class of $w$ is parameterised by the following bijection
\begin{align*}
&h^H \times \mathcal{P}(w)^{C_H(h)}
\times \bigtimes_{z\in \mathcal{C}(w)} 
\left([\gamma_z,z]\operatorname{Yade}\right)^K
\times \bigtimes_{z\in \mathcal{C^{*}}(w)}
K^{ \supp(h_z)\setminus \{\gamma_z\}}
\xrightarrow{1:1} w^W,\\
&(h^{t_a}, \mathcal{P}(w)^c, x,d) \mapsto \prod_{z\in \mathcal{C}(w)\setminus \mathcal{C^{*}}(w)}
\left([\gamma_{z}^{b_a},x_z]\mathcal{E},1_H
\right) \cdot \prod_{z\in \C^{*}(w)}
\left([h_z^{b_a}, \gamma_z^{b_a}, x_z, d_z^{b_a}
]\mathcal{B},h_z^{b_a}\right),
\end{align*}
where $b_a\coloneqq c\cdot t_a$.
\end{theorem}

\begin{proof}
Let $v=(e,g)\in W$. By Theorem \ref{conjugateInWreath}, $w$ and $v$ are
conjugate in $W$ if and only if there exists a $t\in H$ with $g=h^t$ and 
$\mathcal{P}(v)=\mathcal{P}(w)^t$. We first claim that 
\[h^H\times \mathcal{P}(w)^{C_H(h)}\xrightarrow{1:1}\left\{\,
\left(h^t, \mathcal{P}(w)^t\right)\,:\, t\in H\right\},\,
(h^{t_a}, \mathcal{P}(w)^c) \mapsto (h^{ct_a}, \mathcal{P}(w)^{ct_a})\] is
a bijection and first
show injectivity. Fix $1\leq a,b\leq m$ and $c,c'\in C_H(h)$ and assume 
$(h^{ct_a}, \mathcal{P}(w)^{ct_a}) = (h^{c't_b}, \mathcal{P}(w)^{c't_b})$. Then $t_a=t_b$ as
they are representatives of right cosets of $C_H(h)$ in $H$ and hence
$\mathcal{P}(w)^{c} = \mathcal{P}(w)^{c'}$. Surjectivity follows as, for an arbitrary
$t\in H$, there exists $1\leq a \leq m$ and $c\in C_H(h)$ with $t=ct_a$.\\
Now fix $(h^t, \mathcal{P}(w)^t)$ for some $t\in H$. In order to parameterise all
elements $x$ of $W$ with top component $h^t$ and territory decomposition
$\mathcal{P}(w)^t$, we consider a wreath cycle decomposition for each such
element $x = x_1\cdots x_\ell\in W$. Note that
for each wreath cycle $x_i\in S$ in such a decomposition, its load, top component and territory are fixed by
our hypothesis. Thus for each $x_i$, we only need to
consider its base component. By using the maps $\mathcal{E}$ and
$\mathcal{B}$ one proceeds as in Lemma
\ref{lemma-conjugacyclass-wreathcycle}.
\end{proof}

\begin{example}\label{ex:conjugacy}
We use the notation from Example \ref{Example-Wreath-Product-Conjugacy-Problem}
and highlight wreath cycle decompositions by colouring points of $\Gamma$.
We computed a wreath cycle decomposition of $w$ in
Example \ref{Example-Wreath-Product-Wreath-Cycle-Decomp} and $\mathcal{P}(w)$ in 
Example \ref{Example-Wreath-Product-Territory-Partition} using the same colours.
Consider the element $w = (f,h) \in S$
\[
\arraycolsep=0.25em\def\arraystretch{2.2}
\begin{array}{l*{10}{c}}
w \coloneqq \bigl(
& \overset{\highlightB{1}}{(1,2)(3,4)},
& \overset{\highlightB{2}}{(3,4)},
& \overset{\highlightC{3}}{()},
& \overset{\highlightC{4}}{(1,2)},
& \overset{\highlightD{5}}{(1,2,3)},
& \overset{\highlightD{6}}{()},
& \overset{\highlightA{7}}{(1,2)},
& \overset{\highlightZ{8}}{()};
& \overset{\text{top}}{(\highlightB{1,2})(\highlightC{3,4})(\highlightD{5,6})}
&\bigr)
\end{array}
\]
with
\[
\begin{blockarray}{*{10}{c}}
& \scriptstyle \highlightZ{1} & \scriptstyle \highlightZ{2} & \scriptstyle \highlightZ{3} & \scriptstyle \highlightZ{4} & \scriptstyle \highlightZ{5} & \scriptstyle \highlightZ{6} & \scriptstyle \highlightZ{7} & \scriptstyle \highlightZ{8} & \\
\begin{block}{l[*{8}{c}]>{\scriptstyle}r}
& \zero & \zero & \zero & \zero & \zero & \zero & \zero & \zero & \highlightZ{k_1}\\
& \{\highlightA{7}\} & \{\highlightB{1,2}\}, \{\highlightC{3,4}\} & \zero & \zero & \zero & \zero & \zero & \zero & \highlightZ{k_2} \\
\mathcal{P}(w) = & \zero & \zero & \zero & \zero & \zero & \zero & \zero & \zero & \highlightZ{k_3} \\
& \zero & \{\highlightD{5,6}\} & \zero & \zero & \zero & \zero & \zero & \zero & \highlightZ{k_4} \\
& \zero & \zero & \zero & \zero & \zero & \zero & \zero & \zero & \highlightZ{k_5} \\
\end{block}
\end{blockarray}.
\]
Using this colouring to encode the wreath cycles, we write $w$ in a
disjoint wreath cycle decomposition as in Equation
\ref{notation:wr_decomp}, where $w_{i,j,\ell}$ denotes the $\ell$-th wreath cycle
of load $(k_i^K,j)$:
\[
w = \highlightA{w_{2,1,1}} \cdot \highlightB{w_{2,2,1}} \cdot \highlightC{w_{2,2,2}} \cdot \highlightD{w_{4,2,1}}
\]
Further let $w_{i,j,k} = (f_{i,j,k}, h_{i,j,k})$ and
fix points in the territory of each wreath cycle:
\[
\highlightA{\gamma_{2,1,1}} \coloneqq \highlightA{7},\quad
\highlightB{\gamma_{2,2,1}} \coloneqq \highlightB{1},\quad
\highlightC{\gamma_{2,2,2}} \coloneqq \highlightC{3}, \quad
\highlightD{\gamma_{4,2,1}} \coloneqq \highlightD{5}.
\]

Recall the three different top groups:
\[
\begin{array}{lcl}
H_1 \coloneqq& \langle ~ (1,2)(3,4),~ (1,2,3,4),~ (5,6),~ (7,8) ~ \rangle & \cong D_8 \times C_2 \times C_2, \\
H_2 \coloneqq& \langle ~ (1,2)(3,4)(5,6),~ (3,5)(4,6)(7,8) ~ \rangle & \cong C_2 \times C_2, \\
H_3 \coloneqq& \langle ~ (1,2)(3,4)(5,6),~ (7,8) ~ \rangle & \cong C_2 \times C_2
\end{array}
\]
Using the computations from Example
\ref{Example-Wreath-Product-Conjugacy-Problem} 
we compute the cardinality of the conjugacy class $w^{W_i}$, where $W_i
\coloneqq K \wr H_i$ for $1 \leq i \leq 3$.
First note that two factors of the cartesian product occurring in the source of the
bijection defined in Theorem \ref{thm-cc} do not depend on the chosen top group, namely
{\footnotesize
\begin{align*}
&\left\vert
\left([\highlightA{7}, \highlightA{w_{2,1,1}}]\Yade\right)^K \times
\left([\highlightB{1}, \highlightB{w_{2,2,1}}]\Yade\right)^K \times
\left([\highlightC{3}, \highlightC{w_{2,2,2}}]\Yade\right)^K \times
\left([\highlightD{5}, \highlightD{w_{4,2,1}}]\Yade\right)^K
\right\vert \\
&= \left\vert
k_2^K \times k_2^K \times k_2^K \times k_4^K
\right\vert = 6^3 \cdot 8 = 1,728
\end{align*}
}%
and
\begin{align*}
\left\vert
K^{\{\highlightB{1,2}\} \setminus \{\highlightB{1}\}} \times
K^{\{\highlightC{3,4}\} \setminus \{\highlightC{3}\}} \times
K^{\{\highlightD{5,6}\} \setminus \{\highlightD{5}\}}
\right\vert = 24^3 = 13,824.
\end{align*}
We have
\begin{align*}
\left\vert w^{W_1} \right\vert &= \left\vert h^{H_1} \right\vert \cdot \left\vert \mathcal{P}^{C_{H_1}(h)} \right\vert \cdot 1,728 \cdot 13,824 = 2 \cdot 2 \cdot 1,728 \cdot 13,824 = 95,551,488\;, \\
\left\vert w^{W_2} \right\vert &= \left\vert h^{H_2} \right\vert \cdot \left\vert \mathcal{P}^{C_{H_2}(h)} \right\vert \cdot 1,728 \cdot 13,824 = 1 \cdot 2 \cdot 1,728 \cdot 13,824 = 47,775,744\;, \\
\left\vert w^{W_3} \right\vert &= \left\vert w^{W_2} \right\vert\;.
\end{align*}
Let $\Phi$ be the bijection from Theorem \ref{thm-cc} for the wreath
product $W_1$, where we fix the transversal $\{t_1 \coloneqq (), t_2
\coloneqq (2,4)\}$ of the right cosets of $C_{H_1}(h)$ in $H_1$.
We choose an element from the domain of $\Phi$ using the same notation as
in the above Theorem:
\[
t_a \coloneqq t_2, \quad c \coloneqq (7,8) \in C_{H_1}(h),
\]
\[
\begin{blockarray}{*{10}{c}}
& \scriptstyle \highlightZ{1} & \scriptstyle \highlightZ{2} & \scriptstyle \highlightZ{3} & \scriptstyle \highlightZ{4} & \scriptstyle \highlightZ{5} & \scriptstyle \highlightZ{6} & \scriptstyle \highlightZ{7} & \scriptstyle \highlightZ{8} & \\
\begin{block}{l[*{8}{c}]>{\scriptstyle}r}
& \zero & \zero & \zero & \zero & \zero & \zero & \zero & \zero & \highlightZ{k_1}\\
& \highlightA{(2,4)} & \highlightB{(3,4)}, \highlightC{(1,4)} & \zero & \zero & \zero & \zero & \zero & \zero & \highlightZ{k_2} \\
x = & \zero & \zero & \zero & \zero & \zero & \zero & \zero & \zero & \highlightZ{k_3} \\
& \zero & \highlightD{(1,4,3)} & \zero & \zero & \zero & \zero & \zero & \zero & \highlightZ{k_4} \\
& \zero & \zero & \zero & \zero & \zero & \zero & \zero & \zero & \highlightZ{k_5} \\
\end{block}
\end{blockarray}
\]
\[
\begin{blockarray}{*{9}{c}}
& \scriptstyle \highlightZ{2} & \scriptstyle \highlightZ{3} & \scriptstyle \highlightZ{4} & \scriptstyle \highlightZ{5} & \scriptstyle \highlightZ{6} & \scriptstyle \highlightZ{7} & \scriptstyle \highlightZ{8} & \\
\begin{block}{l[*{7}{c}]>{\scriptstyle}r}
& \zero & \zero & \zero & \zero & \zero & \zero & \zero & \highlightZ{k_1}\\
& \begin{array}{c}
	\scriptstyle \highlightB{2} \\
	\bigl( \highlightB{(1,2,3,4)} \bigr),
\end{array}
\begin{array}{c}
	\scriptstyle \highlightC{4} \\
	\bigl( \highlightC{(1,4,3)} \bigr)
\end{array}
& \zero & \zero & \zero & \zero & \zero & \zero & \highlightZ{k_2} \\
d = & \zero & \zero & \zero & \zero & \zero & \zero & \zero & \highlightZ{k_3} \\
& \begin{array}{c}
	\scriptstyle \highlightD{6} \\
	\bigl( \highlightD{(2,3,1)} \bigr)
\end{array}
& \zero & \zero & \zero & \zero & \zero & \zero & \highlightZ{k_4} \\
& \zero & \zero & \zero & \zero & \zero & \zero & \zero & \highlightZ{k_5} \\
\end{block}
\end{blockarray}.
\]
We compute the image of $(h^{t_a}, c, x, d)$ under $\Phi$ by computing each non-trivial factor in the product separately:
\[
\arraycolsep=0.15em\def\arraystretch{1.8}
\begin{array}{r*{10}{c}}
\highlightE{u_{2,1,1}} = \bigl(&
\multicolumn{8}{c}{
[\highlightA{\gamma_{2,1,1}}^{ct_a}, \highlightA{x_{2,1,1}}]\mathcal{E},}
& ()
& \bigr)\\[10pt]
= \bigl(
& \overset{\highlightZ{1}}{()},
& \overset{\highlightZ{2}}{()},
& \overset{\highlightZ{3}}{()},
& \overset{\highlightZ{4}}{()},
& \overset{\highlightZ{5}}{()},
& \overset{\highlightZ{6}}{()},
& \overset{\highlightZ{7}}{()},
& \overset{\highlightE{8}}{(2,4)};
& \overset{\text{top}}{()}
&\bigr),\\[10pt]
\highlightF{u_{2,2,1}} = \bigl(&
\multicolumn{8}{c}{
[\highlightB{h_{2,2,1}}^{ct_a}, \highlightB{\gamma_{2,2,1}}^{ct_a}, \highlightB{x_{2,2,1}}, \highlightB{d_{2,2,1}}^{ct_a}]\mathcal{B},}
& \highlightB{h_{2,2,1}}^{ct_a}
& \bigr)\\[10pt]
= \bigl(
& \overset{\highlightF{1}}{(1,4,2)},
& \overset{\highlightZ{2}}{()},
& \overset{\highlightZ{3}}{()},
& \overset{\highlightF{4}}{(1,2,3,4)},
& \overset{\highlightZ{5}}{()},
& \overset{\highlightZ{6}}{()},
& \overset{\highlightZ{7}}{()},
& \overset{\highlightZ{8}}{()};
& \overset{\text{top}}{\highlightF{(1,4)}}
&\bigr),\\[10pt]
\highlightG{u_{2,2,2}} = \bigl(&
\multicolumn{8}{c}{
[\highlightC{h_{2,2,2}}^{ct_a}, \highlightC{\gamma_{2,2,2}}^{ct_a}, \highlightC{x_{2,2,2}}, \highlightC{d_{2,2,2}}^{ct_a}]\mathcal{B},}
& \highlightC{h_{2,2,2}}^{ct_a}
& \bigr)\\[10pt]
= \bigl(
& \overset{\highlightZ{1}}{()},
& \overset{\highlightG{2}}{(1,4,3)},
& \overset{\highlightG{3}}{(3,4)},
& \overset{\highlightZ{4}}{()},
& \overset{\highlightZ{5}}{()},
& \overset{\highlightZ{6}}{()},
& \overset{\highlightZ{7}}{()},
& \overset{\highlightZ{8}}{()};
& \overset{\text{top}}{\highlightG{(2,3)}}
&\bigr),\\[10pt]
\highlightH{u_{4,2,1}} = \bigl(&
\multicolumn{8}{c}{
[\highlightD{h_{4,2,1}}^{ct_a}, \highlightD{\gamma_{4,2,1}}^{ct_a}, \highlightD{x_{4,2,1}}, \highlightD{d_{4,2,1}}^{ct_a}]\mathcal{B},}
& \highlightD{h_{4,2,1}}^{ct_a}
& \bigr)\\[10pt]
= \bigl(
& \overset{\highlightZ{1}}{()},
& \overset{\highlightZ{2}}{()},
& \overset{\highlightZ{3}}{()},
& \overset{\highlightZ{4}}{()},
& \overset{\highlightH{5}}{(1,4,2)},
& \overset{\highlightH{6}}{(1,2,3)},
& \overset{\highlightZ{7}}{()},
& \overset{\highlightZ{8}}{()};
& \overset{\text{top}}{\highlightH{(5,6)}}
&\bigr).
\end{array}
\]
Then $u \coloneqq (e,g) =
\highlightE{u_{2,1,1}} \cdot
\highlightF{u_{2,2,1}} \cdot
\highlightG{u_{2,2,2}} \cdot
\highlightH{u_{4,2,1}}
\in w^{W_1}$ is a wreath cycle decomposition of
$u = \left[h^{t_a}, c, x, d\right]\Phi$, where $u_{i,j,\ell}$ denotes the
$\ell$-th wreath cycle of load $(k_i^K,j)$.
Note that by construction we have $h^{ct_a} = g$ and $\mathcal{P}(w)^{ct_a} = \mathcal{P}(u)$.
\end{example}
In the following theorem we construct representatives of all conjugacy
classes of elements of $W$. For this, we need to consider orbits under the
natural action of $C_H(h)$ on
$\{\P(w)\, : \, w=(f,h)\in S\}$ for a fixed $h\in H$.  Note that as $K$ need not be
finite, the index set $I$ in the theorem below need not be finite either.
\begin{theorem}
Let $h \in H$ and define $\Omega \coloneqq \{ \mathcal{P}(w) \,:\, w=(f,h)\in S \}$
as the set of all territory decompositions of elements of $S$ with fixed top
component $h$. Fix a set of representatives
$\{\P(w_i) \,:\, i\in I\}$ of the  orbits under the action of $C_H(h)$ on
$\Omega$, where $I$ is some index set. Denote by $\gamma_z$ the minimum of $\terr(z)$
with respect to a fixed total ordering on $\Gamma$. Then the map
\[
\varphi_h: \{\, \P(w_i) \,:\, i\in I\,\} \hookrightarrow W, \,
\P(w)\mapsto \left( \prod_{z\in \C(w)} 
\left(\left[\gamma_z,\, [\gamma_z,z]\Yade\right]\mathcal{E},1_H\right)\right)\cdot
(1_{K^\Gamma},h),
\]
is an injective map and $\im(\varphi_h)$ consists of a system of
representatives of conjugacy classes of elements of $W$ whose top component
is conjugate to $h$ in $H$.\\
In particular, a system of representatives of $W$-conjugacy classes is
given by 
\[
\mathcal{R}(W) = \bigcupdot_{x\in \mathcal{R}(H)} \im(\varphi_x),
\]
where the union ranges over a system of representatives $\mathcal{R}(H)$ 
of $H$-conjugacy classes.
\end{theorem}
\begin{proof}
Note that for $i \in I$ the elements $w_i$ and $v_i \coloneqq [\P(w_i)]\varphi_h$ are conjugate in $W$ by Theorem \ref{conjugationInWreathProd},
since $\P(w_i) = \P(v_i)$ and the top components of $w_i$ and $v_i$ are identical.
We first prove that elements of $\im(\varphi_h)$ for fixed $h\in H$ are
non-conjugate.
Let $i\neq j \in I$. By Theorem \ref{conjugationInWreathProd} $w_i$ and
$w_j$ are conjugate in $W$ if and only if there exists a $t\in H$ with $h^t=h$ and
$\mathcal{P}(w_i)^t=\mathcal{P}(w_j)$, so $t\in C_H(h)$. As
$\mathcal{P}(w_i)$ and $\mathcal{P}(w_j)$ are in
different centraliser orbits we conclude that $w_i$ and $w_j$ are not conjugate
in $W$.\\
Let $(f,h^t)\in W$ for some $t\in H$.
We now show that $(f,h^t)$ is conjugate to an
element in $\im(\varphi_h)$. Note
$(f,h^t)^{(1_{K^\Gamma},t^{-1})} = (e,h)$ for a suitable $e\in K^\Gamma$,
so it suffices to show that $(e,h)$ is conjugate to an element of
$\im(\varphi_h)$. There exists an $i\in I$ and a $c\in
C_H(h)$ with $\mathcal{P}((e,h)) = \P(w_i)^c$, hence $[\P(w_i)]\varphi_h$ is conjugate to
$(e,h)$ in $W$.
\end{proof}
\begin{example} \label{Example-Conjugacy-Representatives}
Recall the groups $H_i$ and $W_i$ from Example \ref{ex:conjugacy}.
The number of conjugacy classes of $H_i$ for $i=1,2,3$ is 
\[
\left\vert \mathcal{R}(H_1) \right\vert	= 20,\,
\left\vert \mathcal{R}(H_2) \right\vert	= 4,\,
\left\vert \mathcal{R}(H_3) \right\vert	= 4
\]
and using the above theorem we compute
\[
\left\vert \mathcal{R}(W_1) \right\vert = 92\,000, \,
\left\vert \mathcal{R}(W_2) \right\vert = 103\,000 \text{ and }
\left\vert \mathcal{R}(W_3) \right\vert = 160\,000.
\]
As an example we demonstrate how this is done for $W_2$.
We first choose a set $\mathcal{R}(H_2)$ 
of representatives of the $H_2$-conjugacy classes as
\[
\{x_1 \coloneqq (),\; x_2 \coloneqq (3,5)(4,6)(7,8),\;
x_3 \coloneqq (1,2)(3,4)(5,6),\; x_4 \coloneqq (1,2)(3,6)(4,5)(7,8)\}.
\]
For each element $x_i \in \mathcal{R}(H_2)$, we 
compute the images under $\varphi_{x_i}$:
\[
\left\vert \im(\varphi_{x_1}) \right\vert = 99\,375,\,
\left\vert \im(\varphi_{x_2}) \right\vert = 1\,625,\,
\left\vert \im(\varphi_{x_3}) \right\vert = 1\,625,\,
\left\vert \im(\varphi_{x_4}) \right\vert = 375.
\]
In particular, these computations show
that there are $99\,375$ conjugacy classes in $W_2$ 
whose elements have trivial top component.
\end{example}
\section{Centralisers in wreath products}
Recall the setting from Hypothesis \ref{setting}.
We first introduce the notion of sparse wreath cycles. These are wreath
cycles with at most one non-trivial base component. It turns out that every
wreath cycle is conjugate in $K^\Gamma\times
\langle 1_H \rangle$ to a sparse wreath cycle 
and we show that one can conjugate every wreath product
element into a sparse wreath cycle decomposition, see Corollary \ref{cor-sparse}.
We use this to parameterise the $W$-centraliser of a product of sparse
wreath cycles which, after conjugation, then parameterises the $W$-centraliser
of an arbitrary wreath product element. 
The structure of $C_S(w)$ for the full monomial group $S=K\wrg \symg$
is described in \cite[Theorem 8]{Ore}.
\begin{definition}
Let $w=(f,h)\in W$ be a
wreath cycle. We call $w$ a \emph{sparse wreath cycle} if there exists a
$\gamma_0 \in \Gamma$ such that $[\gamma]f=1_K$ for all $\gamma\in \Gamma \setminus
\{\gamma_0\}$.
\end{definition}
The concept of sparse wreath cycles is described in Ore \cite[Theorem 7]{Ore}.\\
Note that in a disjoint wreath cycle decomposition of an element $w\in W$ in Theorem
\ref{thm:wreathcycle-decomposition} we have
\[w = (f,h) =
\prod_{i=1}^\ell(\myrestriction{f}{\supp(h_i)},h_i) \quad \cdot \!
\prod_{\gamma \in \fix(h) \cap \terr(w)} (\myrestriction{f}{\gamma}, 1_H)
\in W\]
and the factors
$(\myrestriction{f}{\gamma}, 1_H)$ are sparse wreath cycles for all $\gamma
\in \fix(h)\cap \terr(w)$.

The following corollary shows that every wreath cycle is conjugate to a
sparse wreath cycle and that one can write a $K^\Gamma \rtimes \langle 1_H
\rangle$-conjugate of any wreath
product element as a product of sparse wreath cycles. The following
is a corollary of Theorem \ref{conjugationInWreathProd}.
\begin{cor}\label{cor-sparse}
Let $w=w_1\cdots w_\ell \in S$ be a disjoint wreath cycle decomposition of $w$. Then
there exists an $a\in K^\Gamma\rtimes \langle 1_H\rangle$ such that $w^a =
w_1^a\cdots w_\ell^a$ and $w_i^a$ is a sparse wreath cycle for all $1\leq i
\leq \ell$. This is called a \emph{sparse wreath cycle decomposition} of $w^a$.
\end{cor}
\begin{proof}
For $1\leq i\leq \ell$ let
$w_i = (f_i,h_i)\in S$, choose $\gamma_i \in \terr(w_i)$ and define \[e_i:
\Gamma \to K, \gamma\mapsto 
\begin{cases}
[\gamma_i, w_i]\Yade,& \text{ if }\gamma=\gamma_i\\
1_K,& \text{ else }.
\end{cases}\]
Now, for $1\leq i\leq \ell$ set $v_i \coloneqq (e_i,h_i)$ and $v\coloneqq
v_1\cdots v_\ell$.
Choosing $t=1_H$, we obtain $h^t=h$ and $\mathcal{P}(w)^t = \mathcal{P}(v)$ and the
existence of an $a\in W$ with $w^a=v$ and top component $t$
follows by the proof of Theorem \ref{conjugationInWreathProd}.
\end{proof}

We now turn towards centralisers of elements of $W$.
It is well known that for a single cycle $h\in \sym(\Gamma)$, the
centraliser of $h$ in $\sym(\Gamma)$ is given by
$C_{\sym(\Gamma)}(h) \simeq \langle h \rangle \times \sym(\Gamma\setminus
\supp(h))$. The goal of this section is to give an explicit parametrisation
of the $W$-centraliser of an arbitrary wreath product element $w\in W$ by an
iterated cartesian product.\\
We first observe a relation between the top component of elements of
$C_W(w)$ and the stabiliser of the territory decomposition of $w$.
\begin{lemma}\label{lem:TopStabiliser}
Let $w=(f,h)\in S$ and $a = (s,t)\in C_W(w)$. Then $t \in
\operatorname{Stab}_{C_{\sym(\Gamma)}(h)}(\mathcal{P}(w))$.
\end{lemma}
\begin{proof}
Let $a=(s,t)\in C_W(w)$. Then  $\P(w) = \P(w^a) = \P(w)^t$ and therefore $t\in
\operatorname{Stab}_{\sym(\Gamma)}(\mathcal{P}(w))$ by Corollary
\ref{conjOfTerrDecomp}. It remains to show $t\in C_{\sym(\Gamma)}(h)$. Observe
\[
(f,h) = w = w^a =  \left( \left(s^{-1}\right)^{t}\cdot f^t\cdot s^{h^{-1}\cdot t},
t^{-1}\cdot h\cdot t \right),
\]
which implies $h=h^t$ and the claim follows.
\end{proof}
To describe $C_W(w)$ explicitly, Lemma \ref{lem:TopStabiliser} suggests to
investigate the structure of
$\operatorname{Stab}_{C_{\sym(\Gamma)}(h)}(\mathcal{P}(w))$ further. First,
we restate the group theoretic structure of $C_{\sym(\Gamma)}(h)$ for $h\in
\symg$.
\begin{lemma}[{\cite[Lemma 6.1.8]{Serres}}]\label{lem:centraliser_serres}
Let $h\in \symg$ and $\left\{\mathcal{O}_1,\dots, \mathcal{O}_k\right\}$ be a
system of representatives of
equivalence classes of orbits of $\langle h \rangle$ on $\Gamma$, where two
orbits are equivalent if and only if they have the same cardinality. Then 
\[
C_{\symg}(h) \simeq \bigtimes_{i=1}^k C_{\sym(\mathcal{O}_i)}
\left(\langle h\rangle^{\mathcal{O}_i}\right)\wr \sym \left(\vert \overline{\mathcal{O}_i} \vert \right),
\]
where for all $i=1,\dots, k$ we denote by  $\langle h\rangle^{\mathcal{O}_i}$ the group $\langle
h \rangle$ induces on the orbit $\mathcal{O}_i$ and
$\overline{\mathcal{O}_i}$ denotes the equivalence class of the
representative $\mathcal{O}_i$.
\end{lemma}
As our goal is to describe the centraliser of a wreath product
element explicitly, we require concrete isomorphisms. Our next aim is to give a
constructive version of the above lemma. For this, we start by
investigating what structure the 
preimage of $\sym\left(\vert \overline{\mathcal{O}_i}\vert \right)$ must
have under such an isomorphism.\\
The following definition constructs an element in $\symg$ induced by a
permutation of a set of pairwise disjoint cycles in $\symg$ of the same
order.
\begin{definition}\label{def:Psi_embedding}
Let $I$ be a finite set and, for $i\in I$, let $h_i\in \symg$ 
be pairwise disjoint cycles of the same order and fix
$\gamma_i\in \supp(h_i)$.
Define the map $\Psi$ via
\begin{align*}
&[\{(h_i,\gamma_i)\,:\, i\in I\}, -]\Psi: \sym(I)\hookrightarrow \symg,\\
&\sigma \mapsto \left(\gamma \mapsto 
\begin{cases}
[\gamma_{[i]\sigma}]\, h_{[i]\sigma}^j,& \text{ if }
\gamma = [\gamma_i]\, h_i^j \text{ for some }i\in I,\, j\in \mathbb{Z}_{>0}\\
\gamma,& \text{ else}
\end{cases}
\right).
\end{align*}
\end{definition}
The proof of the following lemma is straightforward.
\begin{lemma}
The map $\Psi$ in the definition above is a monomorphism.
\end{lemma}
The next lemma is a constructive version of Lemma
\ref{lem:centraliser_serres} by using the map $\Psi$ introduced in Definition
\ref{def:Psi_embedding} to describe $C_{\symg}(h)$ for $h=h_1\cdots
h_\ell\in \symg$ in a fixed disjoint cycle decomposition.
We consider a partition of $\Gamma$ whose parts consist of the union of
the supports of cycles $h_i$  of equal order. This partition corresponds to the equivalence
classes $\mathcal{O}_1, \dots, \mathcal{O}_k$ from Lemma \ref{lem:centraliser_serres}.
A necessary condition for an element of $\sym(\Gamma)$ to centralise $h$ is that it
stabilises this partition.
\begin{lemma}\label{lem:centraliser_symg}
Let $h=h_1\cdots h_\ell\in \symg$ be in disjoint cycle decomposition and 
define $\mathcal{O}(h) \coloneqq \{ \, \vert h_i \vert \,:\, 1\leq i \leq
\ell\}$ and for $\mathfrak{o}\in \mathcal{O}(h)$ define $\C(\mathfrak{o},h)\coloneqq 
\{h_i \,:\, \vert h_i\vert = \mathfrak{o},\, 1\leq i \leq \ell \}$.
For each $\mathfrak{o}\in \mathcal{O}(h)$ and each $z\in \C(\mathfrak{o},h)$,
choose $\gamma_z\in \supp(z)$ and define $\Psi_{\mathfrak{o}}: \sym(\C(\mathfrak{o},h))\hookrightarrow
\symg, \sigma\mapsto [((z,\gamma_{z})\,:\, z\in \C(\mathfrak{o},h)),\sigma]\Psi$
as in Definition \ref{def:Psi_embedding}.
Then the elements of $C_{\symg}(h)$ are parameterised by the following
group isomorphism
\begin{align*}
\bigtimes_{\mathfrak{o}\in \mathcal{O}(h)}
\left( \langle (1,\dots, \mathfrak{o})\rangle \wr_{\C(\mathfrak{o},h)} 
\sym( \C(\mathfrak{o},h)\right) \times
\sym(\Gamma\setminus \supp(h)) 
\xrightarrow{\sim} C_{\symg}(h),\\
\left( \left(\left((1,\dots,\mathfrak{o})^{e_z}\right)_{z\in \C(\mathfrak{o},h)},\, \sigma_{\mathfrak{o}}\right)_{\mathfrak{o}\in
	\mathcal{O}(h)} ,\, \pi_0 \right)
\mapsto \left( \prod_{\mathfrak{o}\in \mathcal{O}(h)}
\left(\prod_{z\in \C(\mathfrak{o},h)} h_z^{e_z}\right)\cdot [\sigma_{\mathfrak{o}}]
\Psi_{\mathfrak{o}}\right)\cdot \pi_0,
\end{align*}
where for $z\in \C(\mathfrak{o},h)$, the integer $e_z \in \{0, \dots, \mathfrak{o}-1\}$. 
\end{lemma}
As announced after Lemma \ref{lem:TopStabiliser}, we give
an explicit bijection from an iterated cartesian product into 
$\operatorname{Stab}_{C_{\sym(\Gamma)}(h)}(\mathcal{P}(w))$ which is a
crucial step towards the parametrisation of $C_W(w)$. 
We proceed in a similar way as in Lemma \ref{lem:centraliser_symg}, where we
consider the partition on $\Gamma$ induced by the orders of the disjoint cycles of
$h$. We now translate these concepts from permutation groups to wreath products.
Recall the decomposition of a wreath product element $w$ into disjoint wreath
cycles $w=(f,h)=w_1\cdots w_\ell$ from Theorem \ref{thm:wreathcycle-decomposition}. By Lemma
\ref{lem:load_invariant} for each $a = (s,t) \in C_W(w)$ and each $w_i\in \C(w)$
we have $w_i^a\in \C(w)$ and $\ld(w_i^a)=\ld(w_i)$. 
By Lemma \ref{lem:TopStabiliser} the top element $t\in \symg$ must
centralise $h$ and stabilise $\P(w)$.  This
territory decomposition can be viewed as a refinement of the partition of
$\Gamma$ above.
\begin{lemma}\label{lem:centraliser_top}
Let $w = (f,h) \in S$. For each $z\in \C(w)$
choose $\gamma_z\in \terr(z)$. 
For all $L\in \L(w)$, define \[\Psi_L: \sym(\C(L,w))\hookrightarrow
\symg,\, \sigma\mapsto [\{(h_z,\gamma_z)\,:\, z=(f_z,h_z)\in
\C(L,w)\},\sigma]\Psi,\] where $\Psi$ is as in Definition
\ref{def:Psi_embedding}. Then 
the elements of $\operatorname{Stab}_{C_{\sym(\Gamma)}(h)}(\mathcal{P}(w))$
are parameterised by the following group isomorphism 
\begin{align*}
\bigtimes_{L\in \L(w)}
\left( \langle h_L\rangle \wr_{\C(L,w)} 
\sym( \C(L,w)\right) \times
\sym(\Gamma\setminus \terr(w)) 
\xrightarrow{\sim}
\operatorname{Stab}_{C_{\sym(\Gamma)}(h)}(\mathcal{P}(w)),\\
\left( \left(\left(h_L^{e_z}\right)_{z\in \C(L,w)},\, \sigma_L\right)_{L\in
 \L(w)} ,\, \pi_0 \right)
\mapsto \left( \prod_{L\in \L(w)}
\left(\prod_{z\in \C(L,w)} h_z^{e_z}\right)\cdot [\sigma_L]
\Psi_L\right)\cdot \pi_0,
\end{align*}
where for $L=(k^K,j)\in \L(w)$ we define $h_L \coloneqq (1,\dots, j)$ and
for each $z\in \C(L,w)$ the integer $e_z \in \{0,\dots, \vert h_L\vert-1\}$.
\end{lemma}
As in Corollary \ref{cor-sparse}, every wreath product element $w'\in W$ 
is conjugate to a wreath
product element $w\in W$ in sparse disjoint wreath cycle decomposition. Therefore,
the next theorem is only stated for elements in sparse, disjoint wreath
cycle decomposition as $C_W(w') = C_W(w^a) = C_W(w)^a$ for a suitable 
$a\in W$. This element $a$ can be constructed explicitly by Corollary \ref{cor-sparse}.
We generalise Ore's result \cite[Theorem 8]{Ore} for the full monomial
group to the case $W=K\wrg H$, where $H$ need not be the
full symmetric group and $K$ need not be finite. We remark that in the
parametrisation of the elements of $C_W(w)$ the stabiliser
$\operatorname{Stab}_{C_W(w)}(\P(w))$ arises. This is due to Lemma
\ref{lem:TopStabiliser}.\\
We first introduce some notation. Let $w=(f,h)=\prod_{i=1}^\ell w_i\in W$ 
be in sparse, disjoint wreath cycle decomposition as in Corollary
\ref{cor-sparse}. For each cycle $z\in \C(w)$, fix a point $\gamma_z\in
\terr(z)$ such that $[\gamma]f = 1_K$ for all $\gamma \in \terr(z)\setminus
\{\gamma_z\}$. Moreover, for each load $L\in \L(w)$, choose one
representative cycle $z_L\in \C(L,w)$ and fix $\gamma_L \coloneqq \gamma_{z_L}$.
For any other $z \in \C(L,w)$, fix elements $x_z\in
K$ with $[\gamma_z]f = x_z^{-1}\cdot [\gamma_L]f\cdot x_z$. In particular, 
for every element $\gamma\in \terr(w)$ and any non-negative integer $e$
there exists a unique $z=(f_z,h_z)\in \C(w)$ and $0\leq j < \vert h_z
\vert$ such that $\gamma = \gamma_z^{h^{j-e}}$.
\begin{theorem}\label{thm:structure_centraliser}
Let $w=(f,h)=\prod_{i=1}^\ell w_i\in W$ 
be in sparse, disjoint wreath cycle decomposition.
Then the elements of $C_w(w)$ can be parameterised by the following bijection
$\Phi$ defined below:
\begin{align*}
\Phi&: \left( \bigtimes_{L\in \L(w)} C_K([\gamma_L]f)^{\{\terr(z)\,:\,
z\in \C(L,w)\}}\times K^{\Gamma \setminus \terr(w)}\right)
\times \operatorname{Stab}_{C_H(h)}(\mathcal{P}(w))\xrightarrow{1:1}C_W(w),\\
&(c,t) = \left( \left(\left(c_{L,z}\right)_{L\in \L(w),\, z\in \C(L,w)},\,c_0 \right)
,\left( \prod_{L\in \L(w)}
\left(\prod_{z\in \C(L,w)} h_z^{e_z}\right)\cdot [\sigma_L]
\Psi_L\right)\cdot \pi_0\right) \mapsto (s,t),
\end{align*}
where $t$ is parameterised according to the image of the map
in Lemma \ref{lem:centraliser_top} and  $s:\Gamma\to K$ is
defined by
\[
[\gamma]s = 
\begin{cases}
x_z^{-1}\cdot c_{L,z}\cdot x_{[z]\sigma_L}
, &\text{ if }\gamma = \gamma_z^{h^{j-e_z}} \in \terr(w),\, e_z<j<\vert 
h_z\vert\text{ or } j=0,\\
x_z^{-1}\cdot c_{L,z}\cdot x_{[z]\sigma_L} \cdot \left[\gamma_{[z]\sigma_L}
\right]f, &\text{ if }\gamma = \gamma_z^{h^{j-e_z}} \in \terr(w),\, 
1\leq j \leq e_z,\\
[\gamma]c_0, &\gamma \not\in \terr(w).
\end{cases}
\]
\end{theorem}
\begin{proof}
We omit the proof that $\Phi$ is well-defined in order to concentrate on
the more important property that $\Phi$ is surjective. A proof for
well-definedness can be deduced from the arguments below by reversing them.
We show surjectivity of $\Phi$ by proving that  every element of $C_W(w)$ can be 
decomposed into the components of the domain of $\Phi$.\\
Let $a=(s,t)\in C_W(w)$. Then $t\in C_H(h)$, since $w=w^a = \left((s^{-1})^t\cdot f^t \cdot
s^{h^{-1}t},\, t^{-1}\cdot h\cdot t \right)$. Moreover, $t\in
\operatorname{Stab}_{C_H(h)}(\mathcal{P}(w))$,
as $\mathcal{P}(w) = \mathcal{P}(w^a) = \mathcal{P}(w)^t$. 
By Lemma \ref{lem:centraliser_top}, we
parameterise $t\in\operatorname{Stab}_{C_H(h)}(\mathcal{P}(w))$ as 
\[t = \left( \prod_{L\in \L(w)}
\left(\prod_{z\in \C(L,w)} h_z^{e_z}\right)\cdot [\sigma_L]
\Psi_L\right)\cdot \pi_0,\]
where $e_{z}\in \mathbb{Z}_{\geq 0}$, $\pi_L \coloneqq [\sigma_L]\Psi_L\in
\im(\Psi_L)$, the map $\Psi_L$ is as in Lemma
\ref{lem:centraliser_top}, and $\pi_0 \in
\operatorname{Sym}(\Gamma\setminus \terr(w))$. We now decompose $a$ into
disjoint wreath product elements. For any $L\in \L(w)$ define $t_L
\coloneqq \left(\prod_{z\in \C(L,w)} h_{z}^{e_{z}}\right)\cdot \pi_L$ and
$t_0 \coloneqq \pi_0$. For any $L\in \L(w)$ define $\Omega_L \coloneqq
\bigcupdot_{z\in \C(L,w)} \terr( z)$ and $\Omega_0 \coloneqq \Gamma
\setminus \terr( w )$. Further set $s_L \coloneqq
\myrestriction{s}{\Omega_L}$, $a_L \coloneqq (s_L,t_L)$ and $a_{0}\coloneqq
(\myrestriction{s}{\Omega_0},t_0)$. Then $a =
a_0\cdot \prod_{L\in \L(w)} a_L$ and the $a_0,\, a_L$ are pairwise disjoint. Note that for
all $L\in \L(w)$ we have $\terr(a_L) \subseteq \Omega_L$ and
$\terr(a_0)\subseteq \Omega_0$
which shows $w^a = \prod_{L\in \L(w)} \left( \prod_{z\in \C(L,w)} z
\right)^{a_L}$ since disjoint cycles commute. In particular, since
$w=w^a$ and as the load of a wreath cycle is invariant under conjugation, we
have $\left( \prod_{z\in \C(L,w)} z\right)^{a_L} = \prod_{z\in \C(L,w)}
z^{a_L} = \prod_{z\in \C(L,w)}
z$. Fix $z\in \C(L,w)$ for some $L\in \L(w)$. Then, for
$y\coloneqq [z]\sigma_L$ we have 
$(f_{y}, h_{y})=w_{y}=z^{a_L} = \left((s_L^{-1})^{t_L}\cdot f_{z}^{t_L} \cdot
s_L^{h_{z}^{-1}t_L}, t_L^{-1}\cdot h_{z}\cdot t_L \right)$. Note that
for all $\gamma\in \Gamma$ we must have
\begin{align}\label{eq:centraliser_base}
[\gamma]\left((s_L^{-1})^{t_L}\cdot f_{z}^{t_L} \cdot
s_L^{h_{z}^{-1}t_L}\right)= [\gamma]f_{y}.
\end{align}
Thus
\begin{align*}
[\gamma_{y}]f
&= [\gamma_{y},y]\Yade
= \prod_{j=0}^{\vert h_{y} \vert - 1} [\gamma_{y}] f_{y}^{h_{y}^{-j}}
= \prod_{j=0}^{\vert h_{y} \vert - 1} \left[\gamma_{y}^{h_{y}^j}\right] f_{y} \\
& \stackrel{(\ref{eq:centraliser_base})}{=} \prod_{j=0}^{\vert h_{y} \vert - 1}
\left[\gamma_{y}^{h_{y}^j}\right] \left((s_L^{-1})^{t_L}\cdot f_{z}^{t_L} \cdot
s_L^{h_{z}^{-1}t_L}\right)\\
&=
\prod_{j=0}^{\vert h_{z} \vert - 1} 
\left[\gamma_{z}^{h_{z}^{j - e_{z}}}\right]s_L^{-1} \cdot
\left[\gamma_{z}^{h_{z}^{j - e_{z}}}\right] f_{z} \cdot
\left[\gamma_{z}^{h_{z}^{j - e_{z} + 1}}
\right]s_L \\
&= \left[\gamma_{z}^{h_{z}^{-e_{z}}}\right]s_L^{-1}
\cdot \left(\prod_{j=0}^{\vert h_{z} \vert - 1} \cdot 
\left[\gamma_{z}^{h_{z}^{j - e_{z}}}\right] f_{z}\right) \cdot 
\left[\gamma_{z}^{h_{z}^{\vert h_{z} \vert - 1 - e_{z} + 1}}\right]s_L\\
& = \left[\gamma_{z}^{h_{z}^{-e_{z}}}\right]s_L^{-1}\cdot \left[\gamma_{z}^{h_{z}^{
-e_{z}}},z\right]\Yade \cdot 
\left[\gamma_{z}^{h_{z}^{-e_{z}}}\right]s_L
= \left[\gamma_{z}^{h_{z}^{-e_{z}}}\right]s_L^{-1}\cdot [\gamma_{z}]f \cdot
\left[\gamma_{z}^{h_{z}^{-e_{z}}}\right]s_L,
\end{align*}
which implies
$\left[\gamma_{z}^{h_{z}^{-e_{z}}}\right]s_L\in
C_K([\gamma_{z}]f) \cdot x_{z}^{-1} \cdot x_{y}
= C_K([\gamma_{L}]f) ^ {x_{z}} \cdot x_{z}^{-1} \cdot x_{y} 
= x_{z}^{-1} \cdot C_K([\gamma_{L}]f) \cdot x_{y}$.
By Equation \ref{eq:centraliser_base},
the component $\left[\gamma_{z}^{h_{z}^{-e_{z}}}\right]s_L$
uniquely determines every other component of $s_L$ on 
$\terr(z) = \left\{\gamma_{z}^{h_{z}^0}, \dots,
\gamma_{z}^{h_{z}^{\vert h_{z}\vert - 1}}\right\} =
\left\{\gamma_{z}^{h_{z}^{0 - e_{z}}}, \dots, \gamma_{z}^{h_{z}^{\vert
h_{z} \vert - 1 - e_{z}}}\right\}$,
since for all $j\in\mathbb{Z}$ we can inductively conclude:
\begin{align*}
\left[\gamma_{z}^{h_{z}^{j+1-e_{z}}}\right]s_L
&=\left[\gamma_{y}^{h_{y}^{j}}\right]s_L^{h_{z}^{-1}t_L}
 \left[\gamma_{y}^{h_{y}^{j}}\right] \left( \left(f_{z}^{-1}\right)^{t_L} \cdot s_L^{t_L} \cdot f_{y} \right)\\
&= \left[\gamma_{z}^{h_{z}^{j-e_{z}}}\right]f_{z}^{-1} \cdot \left[\gamma_{z}^{h_{z}^{j-e_{z}}}\right]s_{i}
\cdot \left[\gamma_{y}^{h_{y}^{j}}\right]f_{y} \\
&= \left[\gamma_{z}^{h_{z}^{j-e_{z}}}\right]f^{-1} \cdot \left[\gamma_{z}^{h_{z}^{j-e_{z}}}\right]s_{i}
\cdot \left[\gamma_{y}^{h_{y}^{j}}\right]f.
\end{align*}
Let $c_{L,z}\in C_K([\gamma_{L}]f)$ such that
$\left[\gamma_{z}^{h_{z}^{-e_{z}}}\right]s_L = x_{z}^{-1} \cdot
c_{L,z} \cdot x_{y}$. 
Note that for $j=0$ we have
$\left[\gamma_{z}^{h^{j-e_{z}}}\right]s_L = x_{z}^{-1}\cdot c_{L,z}
\cdot x_{[z]\sigma_L}$.
By the above induction 
\[
\left[\gamma_{z}^{h^{j-e_{z}}}\right]s_L =
\left(\prod_{n=1}^j \left[\gamma_{z}^{h^{j-n-e_{z}}}\right]f^{-1}\right)\cdot 
x_{z}^{-1}\cdot c_{L,z}\cdot x_{[z]\sigma_L} \cdot \left( \prod_{n=1}^j
\left[ \gamma_{[z]\sigma_L}^{h^{j-n}} \right]f \right),
\]
where $0\leq j \leq \vert h_{z}\vert -1$. Recall $[\delta]f^{-1}=1_K$ for all $\delta \in
\supp(h_{y})\setminus\{\gamma_{y}\}$. Hence, for $1\leq j \leq
e_{z}$, we have 
\[
	\left[\gamma_{z}^{h^{j-e_{z}}}\right]s_L = 
	x_{z}^{-1}\cdot c_{L,z}\cdot x_{[z]\sigma_L} \cdot \left[\gamma_{[z]\sigma_L}
	\right]f 
\]
and for $j>e_{z}$ we have 
\begin{align*}
\left[\gamma_{z}^{h^{j-e_{z}}}\right]s_L &= 
[\gamma_{z}]f\cdot x_{z}^{-1}\cdot c_{z}\cdot x_{[z]\sigma_L} \cdot
\left[\gamma_{[z]\sigma_L}\right]f\\
& = x_{z}^{-1}\cdot [\gamma_{L}]f\cdot x_{z} \cdot x_{z}^{-1}
\cdot c_{L,z}\cdot x_{[z]\sigma_L}\cdot 
x_{[z]\sigma_L^{-1}}\cdot [\gamma_{L}]f\cdot x_{[z]\sigma_L}
	=x_{z}^{-1}\cdot c_{L,z}\cdot x_{[z]\sigma_L},	
\end{align*}
which shows surjectivity of $\Phi$. The well-definedness of $\Phi$ follows
by reversing the order of the arguments used.\\
We now show injectivity $\Phi$: Suppose $[(c,t)]\Phi = [(b,r)]\Phi$. Then $t=r$
and by comparing the images of $\gamma_{z}$ under the base components of
$[(c,t)]\Phi$ and $[(b,t)]\Phi$ one obtains $b=c$.
\end{proof}
Note that $\Phi$ in Theorem \ref{thm:structure_centraliser} is not a group homomorphisms.
However, using $\Phi$ we can obtain a generating set for $C_W(w)$.
\begin{cor}\label{extension_centraliser}
Assume the notation of Theorem \ref{thm:structure_centraliser}. Then
\[
\{1\} \to  B_w \xhookrightarrow{\iota}
C_W(w)
\stackrel{\rho}{\twoheadrightarrow} \operatorname{Stab}_{C_H(h)}(\mathcal{P}(w)) \to \{1\}
\]
is a short exact sequence of groups, where 
\[B_w \coloneqq 
\bigtimes_{L\in \L(w)} C_K([\gamma_L]f)^{\{\terr(z)\,:\, 
z\in \C(L,w)\}}\times K^{\Gamma \setminus \terr(w)},\] 
\[
\iota \colon B_w \hookrightarrow C_W(w),\,
c\mapsto \left[\left(c,1_H\right)\right]\Phi,\, \text{ and } \,
\rho \colon C_W(w)\twoheadrightarrow \operatorname{Stab}_{C_H(h)}(\mathcal{P}(w)),\, (s,t)\mapsto t.
\]
In particular, if $B_w = \langle X \rangle $ and 
$\operatorname{Stab}_{C_H(h)}(\mathcal{P}(w)) = \langle Y \rangle $, then
\[C_W(w) = \left\langle [X]\iota \cup \left\{[(1,t)]\Phi \,:\, t\in Y\right\}\right\rangle.\]
\end{cor}
\begin{proof}
It is clear that $\iota$ is a monomorphism and $\rho$ is an epimorphism. We show
$\im(\iota) = \ker(\rho)$, where the inclusion
$\subseteq$ is obvious. Now suppose $a\in \ker(\rho)$. Then $a=(s,1_H)\in C_W(w)$ for
some $s\in K^\Gamma$ and therefore, for $(c,1_H)\coloneqq
[(s,1_H)]\Phi^{-1}$, we obtain $[c]\iota = a$. As $[[(1,t)]\Phi]\rho = t$ for all $t \in Y$,
the claim for the generating set follows as we have an exact sequence of groups.
\end{proof}
\begin{example}
We use the notation from Example \ref{ex:conjugacy} and 
first compute $\vert C_{W_i}(w)\vert$ for $1 \leq i \leq 3$ using
Theorem \ref{thm:structure_centraliser}.
For this we need to conjugate the element $w$ to an element $v$ in sparse
wreath cycle decomposition, say $v \coloneqq w^b$. Note that in this case we
have $C_{W_i}(w)^b = C_{W_i}(v)$.
For example, using $b \coloneqq [1, (1,3,2,4)]\mathcal{E} \cdot [2,
(1,3)(2,4)]\mathcal{E}$ we have
\[
\arraycolsep=0.3em\def\arraystretch{1.8}
\begin{array}{l*{10}{c}}
v \coloneqq w^b = \bigl(
& \overset{\highlightB{1}}{(3,4)},
& \overset{\highlightB{2}}{()},
& \overset{\highlightC{3}}{()},
& \overset{\highlightC{4}}{(1,2)},
& \overset{\highlightD{5}}{(1,2,3)},
& \overset{\highlightD{6}}{()},
& \overset{\highlightA{7}}{(1,2)},
& \overset{\highlightZ{8}}{()};
& \overset{\text{top}}{(\highlightB{1,2})(\highlightC{3,4})(\highlightD{5,6})}
&\bigr)
\end{array}
\]
with
\[
\begin{blockarray}{*{10}{c}}
& \scriptstyle \highlightZ{1} & \scriptstyle \highlightZ{2} & \scriptstyle \highlightZ{3} & \scriptstyle \highlightZ{4} & \scriptstyle \highlightZ{5} & \scriptstyle \highlightZ{6} & \scriptstyle \highlightZ{7} & \scriptstyle \highlightZ{8} & \\
\begin{block}{l[*{8}{c}]>{\scriptstyle}r}
& \zero & \zero & \zero & \zero & \zero & \zero & \zero & \zero & \highlightZ{k_1}\\
& \{\highlightA{7}\} & \{\highlightB{1,2}\}, \{\highlightC{3,4}\} & \zero & \zero & \zero & \zero & \zero & \zero & \highlightZ{k_2} \\
\mathcal{P}(v) = & \zero & \zero & \zero & \zero & \zero & \zero & \zero & \zero & \highlightZ{k_3} \\
& \zero & \{\highlightD{5,6}\} & \zero & \zero & \zero & \zero & \zero & \zero & \highlightZ{k_4} \\
& \zero & \zero & \zero & \zero & \zero & \zero & \zero & \zero & \highlightZ{k_5} \\
\end{block}
\end{blockarray}.
\]
Using this colouring to encode the sparse wreath cycles, we write $u$ in a disjoint sparse wreath cycle decomposition:
\[
v = \highlightA{v_{2,1,1}} \cdot \highlightB{v_{2,2,1}} \cdot \highlightC{v_{2,2,2}} \cdot \highlightD{v_{4,2,1}}
\]
Further let $v_{i,j,\ell} = (e_{i,j,\ell}, h_{i,j,\ell})$. We choose points in the territory of each sparse wreath cycle as in Theorem \ref{thm:structure_centraliser}:
\[
\highlightA{\gamma_{2,1,1}} \coloneqq \highlightA{7},\quad
\highlightB{\gamma_{2,2,1}} \coloneqq \highlightB{1},\quad
\highlightC{\gamma_{2,2,2}} \coloneqq \highlightC{4}, \quad
\highlightD{\gamma_{4,2,1}} \coloneqq \highlightD{5}.
\]
First note that the left iterated cartesian product occurring in the source of the
bijection $\Phi$ defined in Theorem \ref{thm:structure_centraliser} does
not depend on the chosen top group, namely
\begin{align*}
&\left\vert
C_K\left([\highlightA{7}]f\right)^{\left\{\highlightA{\{7\}}\right\}} \times
C_K\left([\highlightB{1}]f\right)^{\left\{\highlightB{\{1,2\}}, \highlightC{\{3,4\}}\right\}} \times
C_K\left([\highlightD{5}]f\right)^{\left\{\highlightD{\{5,6\}}\right\}} \times
K^{\Gamma \setminus \terr(v)}
\right\vert \\
&\quad= \left\vert C_K\bigl((1,2)\bigr) \times C_K\bigl((3,4)\bigr)^2 \times C_K\bigl((1,2,3)\bigr) \times K \right\vert = 4,608\,.
\end{align*}
We have
\begin{align*}
\left\vert C_{W_1}(v) \right\vert &= 4,608 \cdot \left\vert \operatorname{Stab}_{C_{H_1}(h)}(\mathcal{P}(v)) \right\vert = 4,608 \cdot 8 = 36,864\;,\\
\left\vert C_{W_2}(v) \right\vert &= 4,608 \cdot \left\vert \operatorname{Stab}_{C_{H_2}(h)}(\mathcal{P}(v)) \right\vert = 4,608 \cdot 2 = 9,216\;,\\
\left\vert C_{W_3}(v) \right\vert &= \left\vert C_{W_2}(v) \right\vert\;.
\end{align*}

Let $\Phi$ be the bijection from Theorem \ref{thm:structure_centraliser} for the wreath product $W_1$,
where we choose elements $\highlightA{x_{2,1,1}} \coloneqq ()$, $\highlightB{x_{2,2,1}} \coloneqq ()$, $\highlightC{x_{2,2,2}}\coloneqq (1,3)(2,4)$ and $\highlightD{x_{4,2,1}} \coloneqq ()$.
We choose an element from the domain of $\Phi$ using the same notation as
in Theorem \ref{thm:structure_centraliser}
\[
(c, t) = \Bigl(
\bigl(\begin{blockarray}{*{1}{c}}
\scriptstyle \highlightA{\{7\}} \\
(1,2)
\end{blockarray}\bigr),
\bigl(\begin{blockarray}{*{2}{c}}
\scriptstyle \highlightB{\{1,2\}} & \scriptstyle \highlightC{\{3, 4\}} \\
(3,4), & (1,2)(3,4)
\end{blockarray}\bigr),
\bigl(\begin{blockarray}{*{1}{c}}
\scriptstyle \highlightD{\{4,5\}} \\
(1,3,2)
\end{blockarray}\bigr)
\bigl(\begin{blockarray}{*{1}{c}}
\scriptstyle \highlightZ{\{8\}} \\
(1,2,3,4)
\end{blockarray}\bigr),
\begin{blockarray}{*{1}{c}}
\\
(1,3)(2,4)
\end{blockarray}
\Bigr)
\]
and compute the image of $(c, t)$ under $\Phi$.
For this we first need to write the element $t \in \operatorname{Stab}_{C_{H_1}(h)}(\mathcal{P}(w))$ in a suitable decomposition:
\begin{align*}
t &= \highlightA{()^{e_{2,1,1}}} \cdot [\sigma_{2,1}]\Psi_{2,1}\cdot
\highlightB{(1,2)^{e_{2,2,1}}} \cdot \highlightC{(3,4)^{e_{2,2,2}}} \cdot [\sigma_{2,2}]\Psi_{2,2}
\cdot \highlightD{(5,6)^{e_{4,2,1}}} \cdot [\sigma_{4,2}]\Psi_{4,2}
\cdot \pi_0 \\
&= \highlightB{(1,2)} \cdot \highlightC{(3,4)} \cdot [(1,2)]\Psi_{2,2} = \highlightB{(1,2)} \cdot \highlightC{(3,4)} \cdot (1,4)(2,3)\;.
\end{align*}
Now we can compute the base component of $a \coloneqq (s, t) = [(c,t)]\Phi$.
For example, the images of $s$ under
$\highlightB{\terr(w_{2,2,1})} = \highlightB{\{1, 2\}}$
are:
\begin{align*}
\left[\highlightB{\gamma_{2,2,1}}^{h^{0 - e_{2,2,1}}}\right]s
&= \highlightB{[2]}s
= \highlightB{x_{2,2,1}^{-1}} \cdot \highlightB{c_{2,2,1}} \cdot \highlightC{x_{2,2,[1]\sigma_{2,2}}}\\
&= () \cdot (3,4) \cdot (1,3)(2,4)
= (1,3,2,4),\\
\left[\highlightB{\gamma_{2,2,1}}^{h^{1 - e_{2,2,1}}}\right]s
&= \highlightB{[1]}s
= \highlightB{x_{2,2,1}^{-1}} \cdot \highlightB{c_{2,2,1}} \cdot \highlightC{x_{2,2,[1]\sigma_{2,2}}} \cdot \highlightC{[\gamma_{2,2,[1]\sigma_{2,2}}]e}\\
&= () \cdot (3,4) \cdot (1,3)(2,4) \cdot (1,2)
= (1,3)(2,4).
\end{align*}
Repeating this computation for the territory of each wreath cycle and $\Gamma \setminus \terr(v)$ yields
\[
{\small
\arraycolsep=0.15em\def\arraystretch{1.8}
\begin{array}{l*{10}{c}}
a = \bigl(
& \overset{\highlightB{1}}{(1,3)(2,4)},
& \overset{\highlightB{2}}{(1,3,2,4)},
& \overset{\highlightC{3}}{(1,4)(2,3)},
& \overset{\highlightC{4}}{(1,3,2,4)},
& \overset{\highlightD{5}}{(1,3,2)},
& \overset{\highlightD{6}}{(1,3,2)},
& \overset{\highlightA{7}}{(1,2)},
& \overset{\highlightZ{8}}{(1,2,3,4)};
& \overset{\text{top}}{(1,3)(2,4)}
&\bigr)
\end{array}
}%
\]
\end{example}
\section{Performance of an implementation}\label{section-6}
The third author implemented the disjoint wreath cycle decomposition in
the \textsf{GAP} package \emph{WPE} \cite{Rober_Package}.
Building on this, he implemented algorithms using the theory in
this paper for working in finite wreath products $W=K\wrg H$, where $\Gamma$ is
finite and $H\leq \sym(\Gamma)$. The \textsf{GAP}-package
\emph{WPE} provides methods to test whether
two elements of $K\wrg \sym(\Gamma)$ are conjugate in $W$ and, in this case,
computes a conjugating element. Moreover it provides algorithms to compute
representatives of the $W$-conjugacy classes of elements and methods for 
efficient centraliser computations in $W$.\\
To highlight the efficiency of the new methods, we present sample
computations. These were performed on a 1,8GHz IntelCore i5-5350U and are presented in the
following tables. The first column lists the groups we considered, the
second column labelled \textsf{GAP4}
lists the time taken by native GAP 4.11.1 code and the final
column lists the time taken by the package WPE \cite{Rober_Package} loaded
in GAP 4.11.1. For Table (a) we precomputed a set of $100$ random pairs of
conjugate elements (by conjugating $100$ random elements by a further $100$
random elements). We list the average time of computing a conjugating
element. In some cases, the computation using native \textsf{GAP4}-code
for a single computation took too long and the computation was terminated
after the time recorded in column
\textsf{GAP4}. In this case the symbol $>$ indicates that the computation
was terminated. 
In Table (b) we list the average times to compute
the centralisers of $100$ precomputed random elements. Finally,
Table (c) lists the time taken to compute a set of
representatives of the conjugacy classes of elements of the groups $W$ listed
in the first column and the last column contains the number of conjugacy
classes of elements of that group.
\begin{figure}[H]
\centering
{\small
\begin{subfigure}[h!]{0.5\linewidth}
\begin{tabular}{  c  c  c }\toprule
Group & \textsf{GAP4}& WPE\cite{Rober_Package}\\ \midrule
$S_4\wr S_8$ & $<1$s  & $<1$s \\ 
$S_{10}\wr M_{24} $ & $22$s & $<1$s\\  
$S_{25}\wr S_{100}$ & $>40$m & $<1$s \\ 
$\operatorname{SL}(2,2)\wr \operatorname{PSp}(4,3) $ & $<1$s & $<1$s\\ 
$\operatorname{SL}(2,2) \wr \operatorname{PSU}(4,4)$ & $>40$m& $<1$s\\
$\operatorname{PSL}(5,3)\wr \operatorname{PSU}(6,2)$ & $>40$m & $20$s\\ \bottomrule
\end{tabular}\subcaption{Conjugacy problem}
\end{subfigure}\hfill
\begin{subfigure}[h!]{0.5\linewidth}
\begin{tabular}{  c  c  c }\toprule
Group & \textsf{GAP4} & WPE\cite{Rober_Package}\\ \midrule
$S_4\wr S_8$ & $<1$s & $<1$s \\ 
$S_{10}\wr M_{24} $ & $22$s & $<1$s\\ 
$S_{25}\wr S_{100}$ & $>40$m & $<1$s \\ 
$\operatorname{SL}(2,2)\wr \operatorname{PSp}(4,3) $ & $<1$s & $<1$s\\ 
$\operatorname{SL}(2,2) \wr \operatorname{PSU}(4,4)$ & $>40$m &$<1$s\\ 
$\operatorname{PSL}(5,3)\wr \operatorname{PSU}(6,2)$ & $>40$m &$14$s \\ \bottomrule
\end{tabular}\subcaption{Centraliser of elements}
\end{subfigure}\\[.5cm] \newpage
}%
\begin{subfigure}[h]{0.7\linewidth}
\begin{tabular}{  c  c  c c}\toprule
Group &\textsf{GAP4} & WPE\cite{Rober_Package} & $\#$Conjugacy classes\\ \midrule
$\operatorname{SL}(2,2)\wr \operatorname{PSL}(2,7) $ &$<1$s &$<1$s &$216$\\
$S_4\wr S_8$ & $60$s & $<1$s & $6,765$\\ 
$A_5\wr M_{11}$ &$>40$m &$125$s &$15,695$\\ 
$\operatorname{SU}(3,2)\wr A_7$ & $35$m & $22$s & $398,592$\\
$M_{24}\wr S_7$ & $>40$m &$145$s & $9,293,050$ \\
$S_{7}\wr \operatorname{PSL}(2,7)$ &$>40$m &$300$s & $15,342,750$\\ \bottomrule 
\end{tabular}\subcaption{Conjugacy class representatives}
\end{subfigure}
\caption{Time comparisons between native GAP4 code and the WPE package}
\end{figure}
Our experiments show that the computations in the wreath products are
roughly as hard as the corresponding computations in the groups $K$ and $H$.
\section*{Acknowledgements}
We thank Colva Roney-Dougal for discovering Ore's paper \cite{Ore} which
initiated this project. Sebastian Krammer presented the wreath cycle
decomposition in modern language for an algorithm in his M.Sc.\ thesis. We also 
thank Max Horn and Alexander Hulpke for helpful suggestions and
discussions regarding the \textsf{GAP}-implementation. We thank
anonymous referees for their helpful in-depth comments.\\
This is a contribution to  Project-ID 286237555 – TRR 195 -- by the
Deutsche Forschungsgemeinschaft (DFG, German Research Foundation).
The first author acknowledges financial support from the 
RWTH Scholarships for Doctoral Students.

\end{document}